\documentclass[letter,11pt,oneside]{amsart}
\usepackage[latin1]{inputenc}
\usepackage{tikz}
\usepackage{tikz-cd}
\usepackage{xcolor}
\usepackage{cite}
\usepackage{mathtools}
\usepackage{youngtab}
\usepackage{ytableau}
\usepackage{amsthm}
\usepackage[colorinlistoftodos]{todonotes}
\usepackage{fouriernc,amsmath}
\usepackage{graphicx}
\usepackage{subcaption}
\usepackage{fancyhdr}

\ytableausetup{boxsize=1em}

\newcommand{\rep}[2]{
{\Yvcentermath1 \Yboxdim{10pt} \yng(1)_{#1}\yng(1)_{#2}}}
\newcommand{\repg}[4]{{\Yvcentermath1 \young(#1)}_{#2}{\Yvcentermath1 \young(#3)}_{#4}}

\usepackage{fancyhdr}

\usepackage{geometry}
\usepackage{amssymb}
\usepackage{amsthm}
\usepackage{hyperref}
\DeclareMathOperator{\Hom}{Hom}
\DeclareMathOperator{\spec}{Spec}
\DeclareMathOperator{\cok}{Coker}
\DeclareMathOperator{\stab}{STab}
\DeclareMathOperator{\add}{add}

\address{University of Leeds, LS2 9LT, Leeds, United Kingdom}

\email{ll13s4m@leeds.ac.uk}

\newtheorem{theorem}{Theorem}[section]
\theoremstyle{definition}
\newtheorem{definition}[theorem]{Definition}

\newtheorem{remark}[theorem]{Remark}
\newtheorem{corollary}{Corollary}[theorem]

\newtheorem{lemma}[theorem]{Lemma}

\newtheorem{example}[theorem]{Example}

\begin{document}
\title[NCRs for the Discriminant of $G(m,p,2)$]{Non-Commutative Resolutions for the Discriminant of the Complex Reflection group $G(m,p,2)$}
\author{Simon May}

\maketitle
\begin{abstract}
  We show that for the family of  complex reflection groups $G=G(m,p,2)$ appearing in the Shephard--Todd classification, the endomorphism ring of the reduced hyperplane arrangement $A(G)$ is a non-commutative resolution for the coordinate ring of the discriminant $\Delta$ of $G$. This furthers the work of Buchweitz, Faber and Ingalls who showed that this result holds for any true reflection group. In particular, we construct a matrix factorization for $\Delta$ from $A(G)$ and decompose it using data from the irreducible representations of $G$. For $G(m,p,2)$ we give a full decomposition of this matrix factorization, including for each irreducible representation a corresponding a maximal Cohen--Macaulay module. The decomposition concludes that the endomorphism ring of the reduced hyperplane arrangement $A(G)$ will be a non-commutative resolution.
\end{abstract}

\section{Introduction}

Let $G\subseteq GL(n,\mathbb{C})$ be a finite group acting on $\mathbb{C}^n$. The Chevalley--Shephard--Todd theorem shows that the quotient $\mathbb{C}^n/G$ is smooth if and only if $G$ is a complex reflection group, that is $G$ is generated by complex reflections. The group $G$ also acts on $S:=$Sym$_\mathbb{C}(\mathbb{C}^n)$ and if $G$ is a complex reflection group, the invariant ring $R:=S^G$ is isomorphic to a polynomial ring. The \emph{discriminant} $\Delta$ is the image of the hyperplane arrangement $A(G)$ in $\mathbb{C}^n/G$ under the natural projection $\mathbb{C}^n \rightarrow \mathbb{C}^n/G$. In particular $V(\Delta)$ is a singular hypersurface in $\mathbb{C}^n/G$.  The discriminant is given by a polynomial $\Delta \in R$, with coordinate ring $R/(\Delta)$. The \emph{reduced hyperplane arrangement} is defined by a polynomial $z \in S$ and we will decompose the coordinate ring $S/(z)$ as a module over $R/(\Delta)$.
Shephard and Todd classified all the complex reflection groups into two cases; an infinite family, $G(m,p,n)$ with $1\leq m$, $p|m$, $1\leq n$, and 34 exceptional groups. In the 2-dimensional case, Bannai in \cite{bannai1976}, calculated all discriminants of the complex reflection groups. These are shown to be singular curves of type ADE. In particular the coordinate rings of these curve singularities all have a finite amount of (non-isomorphic) Cohen-Macaulay ($CM$) modules, which are listed in \cite{YujiCM}. 

The main theorem of this paper is:

\begin{theorem} \label{MainT} (Thm \ref{first}, Thm \ref{mod odd}, Thm \ref{mod even}, Thm \ref{Theorem:Full})
Let $G=G(m,p,2)$, then all non-isomorphic $CM$-modules of $R /\Delta$ appear at least once in the decomposition of $S/(z)$ as $CM$ modules over $R/(\Delta)$. We also determine a precise decomposition of $S/(z)$ into $CM$ modules over $R/(\Delta)$.
\end{theorem}

This result is of particular interest because an immediate corollary is that End$_{R/(\Delta)}S/(z)$ is a \textit{non-commutative resolution} for $R/(\Delta)$. That is:

\begin{definition} \cite{DITR} Let $A$ be a commutative noetherian ring. Let $M$ be a finitely generated module which is faithful then End$_{A}(M)$ is called a \emph{non-commutative resolution} (NCR) of $A$ if gldim End$_{A}(M) < \infty$.

\end{definition} $S/(z)$ giving a NCR is an application of a result of Auslander, which was presented in the context of artin algebras, see \cite[Section P]{NCRs} for a great survey on NCRs. Recall that a $CM$-local ring $A$ if of finite $CM$-type if it has a finite number of non-isomorphic $CM$-modules.

\begin{theorem}\cite{FiniteGlobal}
Let $A$ be a $CM$-local ring which is of finite $CM$ type then a NCR arises as follows: Let $M_0,...,M_n$ be a list of all non-isomorphic $CM$ modules. Let $P=M_0 \oplus \dots \oplus M_n$, then $\Lambda=End_A(P)$ has finite global dimension and in particular is a NCR for $A$.
\end{theorem}

Theorem \ref{MainT} is an extension of \cite{BFI} where, for a complex reflection group generated by reflections of order 2, it is shown that End$_{R/\Delta}S/(z)$ is a NCR for $R/(\Delta)$. The groups $G(2k,k,2)$ are generated by order 2 reflections. We first introduce the main objects of study; the discriminant, $\Delta$ and the matrix factorization coming from $z$. Then we will detail how we can use the irreducible representations to decompose the matrix factorization using isotypical components. After this, we first tackle the 1-dimensional representations - for any complex reflection group, these give matrix factorizations corresponding to the irreducible components of the discriminant by considering certain orbits of hyperplanes. For the higher dimensional representations we need to find basis elements for the isotypical components for the coinvariant algebra, which in the case of $G(m,p,2)$ are given by higher Specht polynomials, see \cite{ariki1997}. For the groups $G(m,p,2)$ we give a full description of the decomposition of $S/(z)$ as $R/(\Delta)$ $CM$ modules, in particular we describe which irreducible representations corresponds to which $CM$ module.

This result aims to be a next step into a version of the McKay correspondence for complex reflection groups:  McKay first explained in 1979 that if one considers a finite subgroup $G$ of $SL(2,\mathbb{C})$ and constructs the McKay quiver, from only the data of the irreducible representations of $G$, then this quiver, after collapsing double arrows and forgetting direction, yields an affine simply-laced ADE Dynkin diagram. The classical McKay correspondence is then the observation that these graphs are the same as the desingularisation graph of the orbit singularities $\mathbb{C}^2/G$, see e.g. \cite{G-S-V}.

The correspondence was then explained algebraically by Auslander in \cite{ARq}: we call a subgroup $G \subset GL(n,\mathbb{C})$ \emph{small} if it contains no complex reflections. Auslander proved that for a small subgroup $G \subset GL(2,\mathbb{C})$, the Auslander-Reiten (AR) quiver of $CM$ modules over the invariant ring $R$, under the action of $G$ on $\mathbb{C}^2$, coincides with the McKay quiver of $G$. Furthermore, if $G\subset GL(n,\mathbb{C})$ is small and $S$ the polynomial ring in $n$ variables then the skew-group ring $S* G$ is isomorphic to the endomorphism ring End$_R(S)$. In particular $S* G$ has finite global dimension.

Auslander's version of the McKay correspondence can translated into a statement about NCRs: If $G\subset GL(n,\mathbb{C})$ contains no complex-reflections, then $S*G$ is an NCR for $R=S^G$. If $G$ contains complex reflections the skew group ring fails to be a NCR. This result was then was extended in \cite{BFI} for finite true reflection groups, i.e complex reflection groups generated by order $2$ reflections, where End$_{(R/\Delta)}(S/(z))$ is an NCR for  $R/(\Delta)$. Our result extends this to the case of $G(m,p,2)$. Here we note that the $G(m,m,2)$ ($m\neq 2$) and $G(2p,p,2)$ are true reflection groups and so the result overlaps with that of \cite{BFI}.

 The exceptional group case of rank $2$ is still open, but the method would be the same as the $G(m,p,2)$ case, if basis elements for the coinvariant algebra are calculated. Furthermore for $G$ of higher rank the decomposition of $S/(z)$ into $CM$ modules is unclear, since $R/(\Delta)$ is of not of finite $CM$ type. 
 
 We start with some preliminaries and fix some conventions and notation. Chapter 3 will introduce the main object of study, then Chapter 4 will use linear characters to calculate the corresponding modules of the decomposition. We then recap the representation theory of the groups $G(m,p,n)$ and in Section 6 we calculate the decomposition of $G(m,1,2)$. After covering higher Specht polynomials for $G(m,p,2)$, in Section 8 we calculate the decomposition in the different cases. We then finish with an example of a true-reflection group which illustrates the result from \cite{BFI}.

\section{Preliminaries}
\subsection{Reflection groups and Invariant rings} 

The main context of this paper is in invariant theory over $\mathbb{C}$, see \cite{kane2001reflection} for a general reference. Let $V$ a finite dimensional vector space over $\mathbb{C}$.
Recall that the symmetric algebra of $V$ denoted $S(V)$, has the structure of a (graded) commutative $\mathbb{C}$-algebra where $S(V)=\bigoplus_{j=0}^\infty S_j(V)$ and if $\{x_1,...,x_n\}$ is a basis for $V$ then $S(V)$ is a polynomial ring of the form $S(V)=\mathbb{C}[x_1,...,x_n]$. Let $G$ be a finite subgroup of $GL_n(V)$. The action on $G$ on $V$ is extended to $S(V)$ by the following: $\phi \cdot (x_1,...,x_n) = (\phi \cdot x_1)(\cdots)(\phi \cdot x_n)$. The invariant ring is $S(V)^G = \{f \in S(V): \phi \cdot f = f \text{ for all } \phi \in G\}$. We will from now on denote $S(V)$ as $S$. 

\begin{definition} Let $V$ be a $\mathbb{C}$-vector space, then a \textit{complex reflection} is a (diagonalisable) linear isomorphism $s:V \rightarrow V$, which is not the identity such that it fixes a hyperplane pointwise. A group generated by complex reflections is called a complex reflection group.
\end{definition}

\begin{theorem}(Chevalley-Shephard-Todd). \cite{Che55} \label{CST}\\
Let $V$ be a $\mathbb{C}$-vector space and let $G$ be a finite subgroup of $GL(V)$ then $S^G=R \cong \mathbb{C}[f_1,\dots,f_n]$, where $f_i$ are algebraically independent homogeneous polynomials of positive degree if and only if $G$ is a finite complex-refection group.\end{theorem}

A set of polynomials $\{f_1,...,f_n\}$ which satisfy the above are called a set of basic invariants for $G$ A classification of finite irreducible complex reflection groups is given by Shephard-Todd:
\begin{theorem}[Shephard-Todd]\cite{FURG} \\
All irreducible complex reflection groups fall into one of the following families:
\begin{itemize}
    \item The infinite family $G(m,p,n)$, where $m,p,n \in \mathbb{N}\setminus\{0\}$, $p|m$ and $(m,p,n)\neq (2,2,2)$

    \item The exceptional complex reflection groups $G_4,...,G_{37}$.
\end{itemize}
\end{theorem}

For the rest of the paper the following notation will be reserved: $G$ a complex reflection group of rank $n$, $S$ the polynomial ring $\mathbb{C}[x_1,...,x_n]$ on which $G$ acts and $R$ is the invariant ring $S^G$ under this action.

\subsection{Isotypical decomposition}
Recall that $S=\mathbb{C}[x_1,...,x_n]$ can be viewed as a graded vector space over $\mathbb{C}$. Since $G$ acts on $S$, it is a representation of $G$ and thus it has an isotypical (or canonical) decomposition. The following can be found in more detail in \cite{scott1996linear} Section 2.6. Stanley also studied the Isotypical decomposition of $S$ in \cite{stanley1979}, with more of a view towards the structure of the components as modules over $R$.

Since $G$ is a finite group, it has finitely many irreducible representations (up to isomorphism) and finite distinct characters determining these irreducible representations. Let $\chi_1,...,\chi_r$ be the list of characters that determine the irreducible representations $W_1,...,W_r$ respectively. Consider a decomposition of $S= \bigoplus_{i \in \mathbb{N}} T_i$ into irreducible representations. Then let $V_i$ be the direct sum of those $T_i$ which are isomorphic to $W_i$. Then:

\[S = V_1 \oplus \cdots \oplus V_r\]

$S$ has been decomposed into its canonical decomposition, and is a unique decomposition as shown in \cite{scott1996linear} Theorem 8. 
We call the $V_i's$ the \textit{isotypical} components of the corresponding character $\chi_i$. We will normally denote them by $S^G_\chi$. Note that if $triv$ denotes the trivial character, then $S^G_{triv}=S^G.$ The elements of $S^G_\chi$ are called the \textit{relative invariants} of $\chi$. 

\subsection{The discriminant of the group action}
 \label{Disc}
Let $G$ be a finite complex-reflection group acting on $V$, denote by $A(G)$ the set of reflecting hyperplanes of $G$. If $G$ is obvious in context, this is denoted by $A$. Using Maschke's theorem, one can show that given a hyperplane $H \in A$ the set of all elements that fix $H$ pointwise is a cyclic group. Let $e_H$ denote the order of the cyclic subgroup fixing $H$ and $\alpha_H$ be the defining linear equation for $H$ in $S$. 
We define two important polynomials for our study, which are obtained from the certain relative invariants of the linear characters of $G$: See \cite{Arrangments} Section 6  for a more detailed account of the following definitions and lemma. Fix $G$ as a finite complex reflection group and define:
\[z= {\displaystyle \prod_{H\in A}^{} \alpha_H}\]
and
\[j={\displaystyle \prod_{H\in A}^{} \alpha_H^{e_H-1}},\]
$z$ and $j$ are the relative invariants of the linear characters $\det$ and $\det^{-1}$ respectively. Note that this means they are a basis for the isotypical components of $\det$ and $\det^{-1}$ respectively. 
\begin{definition}
    Let $G$ be a complex-reflection group acting on $V$ with basis $\{x_1,...,x_n\}$, then the \textit{discriminant} $\delta$ of the group action is given by:
    \[\delta(x_1,...,x_n) = jz = {\displaystyle \prod_{H\in A}^{} \alpha_H^{e_H}} \]
\end{definition}
\begin{lemma}\cite{Arrangments} Lemma 6.44.
The discriminant of the group action is an invariant under the action of $G$, i.e $\delta \in R$.
\end{lemma}

\begin{remark}
 
Since the discriminant is an invariant it can be expressed as a polynomial in a set of basic invariants. Fixing $\mathcal{F}=\{f_1,...,f_n\}$ a set of basic invariants, $\delta$ is also a polynomial in $\mathcal{F}$. We define a polynomial also called the \textit{discriminant} $\Delta(X_1,...,X_n,\mathcal{F})$ such that $\Delta(f_1,...,f_n,\mathcal{F})= \delta(x_1,...,x_n).$
\end{remark}

\begin{remark}

Since the polynomial $z$ is a relative invariant for $\det$, if $w \in V_i$, then $zw\in V_i \otimes \det$.
\end{remark}
\subsection{Matrix factorizations and Cohen-Macaulay modules}

The main tool used to study $S/(z)$ is Eisenbud's matrix factorization theorem, which formalises the connection between $CM$-modules and matrix factorizations.

Recalling the definition of a $CM$ module:

\begin{definition}
Let $A$ be a Noetherian commutative ring, a finitely generated $A$-module $M$ is called a \textit{Cohen-Macaulay} ($CM$) module if the depth of $M_\mathfrak{p}$ is equal to the Krull dimension of $A_\mathfrak{p}$ for any prime $\mathfrak{p}$ in $\mathrm{Spec}(A)$. A ring $A$ is said to be a $CM$ ring if it is a $CM$-module over itself. 
\end{definition}

In particular the ring of invariants of a complex reflection group, $S^G$ is isomorphic to a polynomial ring and thus is a finitely generated $CM$ ring.  The context for this section is as follows: Let $A$ be a ring of the form $A=B/I$ in which $B$ is a regular local ring and $I$ is a principal ideal generated by an element $f\neq 0$. Let $\mathfrak{C}(A)$ be the category of Cohen-Macaulay modules over the ring $A$.

Let $\phi, \psi$ be $n\times n$ matrices with entries in $B$ such that,
 \[ \psi \cdot \phi = f \cdot1_{B^{n}} \hspace{1cm} \text{and} \hspace{1cm} \phi \cdot \psi = f \cdot1_{B^{n}} \]

\begin{definition} A pair $(\phi,\psi)$ with entries in $B$, which satisfy the above properties is called a \emph{matrix factorization} of $f$. 
\end{definition}To build the category of matrix factorizations of a hypersurface, we require the notation of morphisms between matrix factorizations.

\begin{definition}
Let $(\phi_1,\psi_1)$ and $(\phi_2,\psi_2)$ be matrix factorizations of $f$. Then a morphism between them is a pair of matrices $(\alpha,\beta)$ such that the following diagram commutes:

\begin{center}
    
\begin{tikzcd}
B^{n_1}  \arrow[d,"\alpha"] \arrow[r,"\phi_1"]& B^{n_1} \arrow[d,"\beta"] \arrow[r,"\psi_1"] & \arrow[d,"\alpha"] B^{n_1} \\ B^{n_2} \arrow[r,"\phi_2"] & B^{n_2} \arrow[r,"\psi_2"] & B^{n_2}     
\end{tikzcd}
\end{center}\end{definition} 

Denote by $MF_B(f)$ the category of matrix factorizations. With the following definition of a direct sum $MF_B(f)$ is an additive category.

\begin{definition}
Let $(\phi_1,\psi_1)$ and $(\phi_2,\psi_2)$ be matrix factorizations in $MF_B(f)$ then:

\[(\phi_1,\psi_1) \oplus (\phi_2,\psi_2) = ( \begin{bmatrix} \phi_1 & 0 \\ 0 & \phi_2\end{bmatrix},\begin{bmatrix} \psi_1 & 0 \\ 0 & \psi_2\end{bmatrix}) \] 
\end{definition}
\begin{definition} Two matrix factorizations are \emph{equivalent} if there is a morphism $(\alpha,\beta)$ in which $\alpha,\beta$ are isomorphisms. \end{definition}

\begin{theorem}(Eisenbud's matrix factorization theorem) \cite{EisenbudHOMO} \\
If $A= B/(f)$ is a hypersurface then the functor $Coker(\phi,\psi) = Coker(\phi)$ induces an equivalence of categories $\underline{MF}_B(f):=MF_B(f)/\{(1,f)\} \simeq \mathfrak{C}(A)$. \end{theorem}

\begin{remark}\label{grade}
Eisenbud defined matrix factorizations for a complete local hypersurface ring. Instead we can alter the definition to graded hypersurface rings. Let $A = \sum_{i=0}^\infty A_i$ be an $\mathbb{N}$- graded $CM$ ring. Define $\mathfrak{grC}(A)$ (resp $\mathfrak{grM}(A)$) to be the category of graded $CM$ modules over $A$ (resp. finitely generated graded $A$ modules) and graded homomorphisms preserving degree (resp. graded $A$-homomorphisms preserving degree). Let $M \in \mathfrak{grM}(A)$, and $n \in \mathbb{Z}$, define $M(n)\in \mathfrak{grM}(A)$ by $M(n)_i=M_{n+i}$. We call the operation $(-)$ degree shifting. Two graded modules $M,N \in \mathfrak{grM}(A)$ are isomorphic up to degree shifting if $M(n) \cong N$ in $\mathfrak{grM}(A)$ for some $n \in \mathbb{Z}$. Chapter 15 of \cite{YujiCM} proves that $\mathfrak{grC}(A)$ is of finite $CM$ type if and only if $CM(\hat{A})$ is. Since the context of the paper is a polynomial ring - a graded $CM$ ring, we will be using this case. For ease of notation, if $A$ is a graded ring, $CM(A)$ will also refer to the graded $CM$ modules. 
\end{remark}

\section{Structure of \(S/(z)\) as an \(R/(\)\(\Delta\)\()\)  module}\label{Structure}

Recall that a complex reflection group $G\subseteq GL(n,\mathbb{C})$  acts on $S=\mathbb{C}[x_1,...,x_n] $. Let $f_1,...,f_n$ be a set of basic invariants under this action on $S$ and $R$ be the invariant ring. Let $(R_+) = (f_1,...,f_n)$ and let $S/(R_+)$ be the coinvariant algebra $S/(f_1,...,f_n)$. During the proof of Theorem \ref{CST} in \cite{Che55} Chevalley shows that as a graded $R$-module $S$ can be decomposed as:
\[ S \cong R \otimes_K S/(R_+),\]

and that as $KG$-modules:
\[S \cong R \otimes_K KG \]

In particular $S$ is a free $R$-module. Since $S$ is a free $R$-module the matrix representing multiplication by $z$ in $S$ will have entries in $R$. Let $d$ (resp $e$) be the degree of $z$ (resp $j$).

As a short exact sequence over $R$, multiplication by $z$ can be expressed as:
\begin{center}
    \begin{tikzcd}
    0 \arrow[r] & S(-d) \arrow[r, "z"] &  S \arrow[r]& S/(z) \arrow[r] & 0 
\end{tikzcd}
\end{center}
\begin{remark}

Since $S$ is a free $R$ modules this can be written as: 
\begin{center}
    \begin{tikzcd}
    0 \arrow[r] & R^{|G|}(-d) \arrow[r, "z"] &  R^{|G|}\arrow[r]& R^{|G|}/(z) \arrow[r] & 0 
\end{tikzcd}
\end{center}
\end{remark}
 
Then using the decomposition into isotypical components:

\begin{center}

\begin{tikzcd}
    0 \arrow[r] & \bigoplus_\chi S^G_\chi(-d) \arrow[r, "z"] &  \bigoplus_\chi S^G_\chi \arrow[r]& (\bigoplus S^G_\chi)/(z)\arrow[r] & 0 
\end{tikzcd}
\end{center}
and so for each $\chi \in \mathrm{irrep}(G)$ we get:
\begin{center}
\begin{tikzcd}
    0 \arrow[r] & S^G_\chi(-d) \arrow[r, "z|_{\chi}"] &   S^G_{\chi \otimes \det} \arrow[r]& ( S^G_{\chi \otimes \det})/(z)\arrow[r] & 0 
\end{tikzcd}
\end{center}
Where $z|_{\chi}$ is the restriction of $z$ on $S^G_\chi$.

Then we consider on which of the different components the restriction map of  $z$ are defined. This gives us, for every $\chi \in$ irrep$(G)$ the following  matrix factorization of $\Delta$: 

\begin{center}
\begin{tikzcd}
     S^G_{\chi\otimes \det}(-d-e) \arrow[r, "j|_{\chi \otimes \det}"] &  S^G_{\chi }(-d) \arrow[r,"z|_{\chi}"]& S^G_{\chi \otimes \det} 
\end{tikzcd}
\end{center}
Denote by $e_\chi$ the dimension of $\chi$, and note that $e_{\chi \otimes \det}=e_{\chi}$, then $S^G_\chi \cong R^{e_\chi^2}$.
\begin{center}
    
\begin{tikzcd}
     R^{e_\chi^2}(-d-e) \arrow[r, "j|_{\chi \otimes \det}"] &  R^{e_\chi^2} (-d) \arrow[r,"z|_\chi"]& R^{e_\chi^2} 
\end{tikzcd}
\end{center}

On the component $S^G_\chi$ let $M_\chi:=$Coker$(z|_\chi)= S^G_{\chi \otimes \det}/z$. Then we have $S/(z)\cong \bigoplus_{\chi \in irrep(G)} M_\chi$. The $M_\chi$ are Cohen-Macaulay over $R/(\Delta)$ since they are Cokernels of matrix factorizations.

Since $(z|_{\chi},j|_{\chi \otimes \det})$ are matrix factorizations for all $\chi \in \mathrm{irrep}(G)$, defining $N_\chi :=$Coker$(j^\chi)= S^G_{\chi \otimes \det^{-1}}/(j)$, we get the exact sequence:
\begin{center}

\begin{tikzcd}
     0 \arrow[r] & \mathrm{Coker}(j|_{\chi \otimes \det},z|_\chi) \arrow[r] & R^{e_\chi^2} \arrow[r] & \mathrm{Coker}(z|_{\chi},j|_{\chi \otimes \det}) \arrow[r] & 0   
\end{tikzcd}
\end{center}

Thus $\mathrm{syz}^1_RM_\chi \cong N_{\chi \otimes \det}$. 

\begin{remark}
By Eisenbud's matrix factorization theorem, the above in terms of matrix factorizations is: We decompose the matrix factorization $(z,j)$ into the matrix factorization $\bigoplus_{\chi \in irrep(G)}(z|_{\chi},j|_{\chi})$ 
\end{remark}
\begin{lemma}
Let $R/(\Delta)$ be of finite $CM$-type, If $S/(z)$ is a representation generator for $R/(\Delta)$ then $S/(j)$ is also a representation generator.
\end{lemma} 
\begin{proof}
Let $M$ be an indecomposable $CM$-module over $R/(\Delta)$. Since $S/(z)$ is a representation generator for $CM(R/(\Delta))$ then since $syz^1_RM$ is $CM$, it is in $\add_{R/(\Delta)}S/(z)$. In particular there exists $\chi \in \mathrm{irrep}(G)$ such that $syz^1_RM \in \add(M_\chi)$. Thus $\mathrm{syz}^1_R \mathrm{syz}^1_R M \cong M \in \add(N_{\chi \otimes \det})$. Hence $M \in \add(S/(j))$ 
\end{proof}

\section{Linear Characters of reflection groups}\label{lin}
Recalling the notation used: $V$ is a $n$-dimensional complex vector space which is endowed with an action of $G \subset GL(V)$, Ref$(G)$ the set of reflections of $G$, $S=S(V)= \mathbb{C}[x_1,...,x_n]$ and $R=S^G$. See \cite{broue} for a more in depth review of the following. Recalling from Section \ref{Disc} that $\alpha_H$ was such that, for $H \in A(G)$, $H=\mathrm{ker}(\alpha_H)$. Let $G(H)$ be the subgroup of $G$ which fixes $H$ and $e_H:= |G(H)|$. Denote by $A/G$ the set of orbits of hyperplanes under $G$, let $\mathfrak{O}$ be an element of $A/G$. Define:
\[j_\mathfrak{O}:=\prod_{H\in \mathfrak{O}}\alpha_H\]

We can then define a linear character $\theta_\mathfrak{O} \in \Hom(G,\mathbb{C}^{\times})$ such that:

\[g(j_\mathfrak{O})=\theta_\mathfrak{O}(g)(j_\mathfrak{O}) \text{ for all g} \in \text{G}\]

$\theta_\mathfrak{O}$ then has the following property: Let $s \in \textrm{Ref}(G)$ then:
\[\theta_\mathfrak{O}(s)= \begin{cases} \det(s) & \text{ if s $\in G(H)$ for some $H \in \mathfrak{O}$} \\ 1 & \text{ otherwise} \end{cases}\]

Note, $\det(s)$ is the determinant of the action of $s$ on the $V$. Thus if we take $z:=\prod_{H \in \mathcal{A}} \alpha_H$ we get $g(z)=\det(g)(z)$, since $G$ is generated by reflections.
To describe linear characters it is enough to restrict to the fixer groups of the hyperplanes. 

\begin{theorem}\cite[Theorem 4.12]{broue} 
Let $\theta \in \Hom(G,\mathbb{C}^{\times})$ and for a hyperplane $H$, denote by $\mathfrak{O}$ the orbit of $H$. Then there is a unique integer $m_{\mathfrak{O}}(\theta)$ such that:

\[\text{\emph{Res}}^G_{G(H)} \theta=\text{\emph{det}}^{m_\mathfrak{O}(\theta)}  \text{  with the condition } 0\leq m_\mathfrak{O}(\theta) <|G(H)|.\]
$\text{i.e for all g} \in \text{G, } \text{\emph{Res}}^G_{G(H)} \theta(g)=\det^{m_\mathfrak{O}(\theta)}(g)$.

Then set $j_\theta:= \prod_{\mathfrak{O} \in \mathcal{A}/G} j_{\mathfrak{O}}$, then $S^G_\theta=Rj_\theta$. We call $j_\theta$  the relative invariant of $\theta$.
\end{theorem}

\begin{proof}
The proof is essentially the fact the $\det$ generates $\Hom(G(H),\mathbb{C}^*)$. The proof for  $S^G_\theta=Rj_\theta$ can be found in \cite[6.37]{Arrangments}.   
\end{proof}

Consider $\Delta_\mathfrak{O}:=j_\mathfrak{O}^{|G(H)|}=\prod_{H\in \mathfrak{O}}\alpha_H^{|G(H)|}$.

\begin{lemma}
$\Delta_\mathfrak{O}$ is invariant under the action of $G$.
\end{lemma}
\begin{proof}
This follows from the above discussion.
\end{proof}
Note that $\prod_{\mathfrak{O} \in A/G}\Delta_{\mathfrak{O}} = \Delta$. The following shows which isotypical components, $\Delta_{\mathfrak{O}}$ is obtained from. Note that the action $-\otimes \det$ sends linear characters to linear characters. Let $\theta \in \text{Hom}(G,\mathbb{C^*})$ then $m_\mathfrak{O}(\theta\otimes \det) =m_\mathfrak{O}(\theta)+1 \text{ mod } e_H$. This is easily seen since $m_\mathfrak{O}(\det)=1$ for all orbits $\mathfrak{O}$. Since we are interested in viewing $(z,j)$ as a matrix factorization of the discriminant, we can use this to find the parts of the matrix factorization on the isotypcial components for the linear characters. Recall that $z|_\theta$ denotes the map restricted on the isotypical component of type $\theta$ of $S$.

\[\begin{tikzcd} S^G_\theta \arrow[r,"z|_{\theta}"] & S^G_{\theta \otimes \det}  .
\end{tikzcd}\]
Which, since $\theta \in \text{Hom}(G,\mathbb{C}^{\times})$, is:

\[\begin{tikzcd} Rj_\theta \arrow[r,"z|_{\theta}"] & Rj_{\theta \otimes \det} .
\end{tikzcd}\]
Thus we need to calculate $zj_\theta$.

\begin{lemma}\label{comp}
The cokernel of the map $z|_{\theta_\mathfrak{O}^{e_H-1}}$ is $R/{\Delta_\mathfrak{O}}$.
\end{lemma}
\begin{proof}
The relative invariants of $\theta^{e_H-1}_\mathfrak{O}$ and $\theta^{e_H-1}_\mathfrak{O} \otimes \det$ are given by:

\[j_{\theta_\mathfrak{O}^{e_H-1}}= j_\mathfrak{O}^{e_H-1}=\prod_{H \in \mathfrak{O}}\alpha_H^{e_H-1} \text{ and } j_{\theta_\mathfrak{O}^{e_H-1} \otimes \det} =\prod_{\mathfrak{q} \in (\mathcal{A}/G) \setminus \{\mathfrak{O}\}}j_\mathfrak{q}= j_{\theta \otimes \det}.\]
i.e $j_{\theta_\mathfrak{O}^{e_H-1} \otimes \det}$ contains all the hyperplanes that are not in the orbit $\mathfrak{O}$.

 Thus,

\[zj_{\theta_\mathfrak{O}^{e_H-1}}=j_\mathfrak{O}^{e_H}\prod_{\mathfrak{q} \in (\mathcal{A}/G) \setminus \mathfrak{O}}\alpha_{\mathfrak{q}}=\Delta_\mathfrak{O}j_{\theta \otimes \det}. \]

\end{proof}

\begin{theorem}\label{comp2}
For all orbits $\mathfrak{O}$, $R/(\Delta_\mathfrak{O})$ is a direct summand of $S/(z)$ as a $R/(\Delta)$ module.
\end{theorem}
\begin{proof}
Since $S\cong \bigoplus_{\chi \in Irr(G)}S^G_\chi$ and $(z,j)$ is a matrix factorization for $\Delta$, the discussion above shows the map $z$ on the Isotypical component $S^G_{\theta_\mathfrak{O}^{e_H-1}}$ gives the component $\Delta_\mathfrak{O}$.
\end{proof}

\section{\(G(m,p,n)\)}
 The representation theory of the group $G(m,p,n)$ is an extension of the representation theory of the symmetric group. In particular we can calculate a basis for the isotypical components of the coinvariant algebra, which in turn allows calculation of $M_\chi$ from Section \ref{Structure}. 
 
\begin{definition}
The abstract group $G(m,1,n)$, with $m,n \geq 1$, is identified with the wreath product $(\mathbb{Z}/m\mathbb{Z})\wr S_n$ and $G(m,p,n)$ are normal subgroups of $G(m,1,n)$, where $p$ must divide $m$ and define $q:=\frac{m}{p}$.
\end{definition}

As a complex reflection group $G(m,p,n)$ can be described as follows: Let $\xi_m$ be a $m$-th root of unity then $G(m,p,n)$ is the group of matrices of the form $PD$ where $P$ is a permutation matrix and $D$ is a diagonal matrix which entries are $\xi_m$ and $\det(D)^{\frac{m}{p}}=1$. Generators and relations can be found in \cite{ariki1994hecke}. There are $n$ generators, which we denote $s_1,...,s_n$.

\begin{example}

For the group $G(m,1,2)$ acting on $\mathbb{C}[x,y]$ the invariants are calculated to be:  $\sigma_1=(xy)^m$ and $\sigma_2=x^m+y^m$, see for example \cite{bannai1976}. The defining equations of the hyperplanes of $G(m,1,2)$ are \cite{Arrangments} Section 6.4:

\[x,y,(x-\xi_m^iy) \hspace{1cm} \text{for  $i \in \{0,1,...,m-1\}$ }\]

Thus the hyperplane arrangement for the group $G(m,1,2)$ is:
\[z=xy(x^m-y^m).\]
The generators of the group $G(m,1,2)$ are 
\[s_1=\begin{bmatrix}\xi_m & 0 \\ 0 & 1 \end{bmatrix}, s_2=\begin{bmatrix}0 &1 \\ 1 & 0 \end{bmatrix}.\]
The hyperplanes $H_x:=\ker(x)$ and $H_y:=\ker(y)$ are in the same $G$-orbit, for example by applying $s_2$. The other hyperplanes are in another distinct orbit. The jacobian and discriminant are given as: 
\[j=(xy)^{m-1}(x^m-y^m),\]
\[\Delta={(xy)}^m(x^m-y^m)^2=\sigma_1(\sigma_2^2-4\sigma_1).\]

\end{example}

\begin{example}

For the subgroups $G(m,p,2)$, $p\neq m$ the invariants are $\sigma_1=(xy)^q$ and $\sigma_2=x^m-y^m$ with:

\[z=xy(x^m-y^m),\]
\[j=(xy)^{q-1}(x^m-y^m),\]
\[\Delta={(xy)}^q(x^m-y^m)^2=\sigma_1(\sigma_2^2-4\sigma_1^p).\]
\end{example}

\begin{example}

When $p=m$ the invariants are $\sigma_1=xy$ and $\sigma_2=x^m+y^m$ with:

\[z=(x^m-y^m)\]
\[j=(x^m-y^m)\]
\[\Delta=(x^m-y^m)^2=\sigma_2^2-4\sigma_1^m.\]
\end{example}
In these examples the invariant ring is $R=\mathbb{C}[\sigma_1,\sigma_2]$. Using Section \ref{lin} we can calculate the modules corresponding to the isotypical components of the linear characters. $G(m,1,2)$ has $2$ orbits of hyperplanes. The hyperplanes given by $x, y$ are in the same orbit, call this $\mathfrak{O}$. The rest of the hyperplanes are in the other, $\mathfrak{q}$. Then, $\Delta_\mathfrak{O}= (xy)^m$ and $\Delta_\mathfrak{q}=(x^m-y^m)^2$. Since these are invariant, they can be expressed in terms of the basis invariants $\sigma_1,\sigma_2$. $\Delta_\mathfrak{O}=\sigma_1$ and $\Delta_\mathfrak{q}=\sigma_2^2-4\sigma_1$. From Lemma \ref{comp} and Theorem \ref{comp2}, $M_{\theta_{\mathfrak{O}}^{e_{H_x}-1}}= R/(\sigma_1)$ and $M_{\theta_{\mathfrak{q}}^{e_{H_{x-y}}-1}}= R/(\sigma_2^2-4\sigma_1)$. 

The subgroups $G(m,p,2) \subseteq G(m,1,2)$ fall into $2$ cases. When $p$ is odd, it is similar to the above, there are 2 orbits $\mathfrak{O}$ and $\mathfrak{q}$ where $\Delta_\mathfrak{O}=(xy)^q =\sigma_1$ and $\Delta_\mathfrak{q}=(x^m-y^m)^2=x^{2m}+y^{2m}-2x^my^m=\sigma_2^2-4\sigma_1^p$. When $p$ is even there are $3$ orbits $\mathfrak{O}_1,\mathfrak{O}_2,\mathfrak{O}_3$, where $\Delta_{\mathfrak{O}_1}= \sigma_1$, $\Delta_{\mathfrak{O}_2}= \sigma_2-2\sigma_1^{\frac{p}{2}}$,$\Delta_{\mathfrak{O}_3}= \sigma_2+2\sigma_1^{\frac{p}{2}}$. With the notation from Section \ref{lin} we can combine the different orbits and show

\[M_{\theta_\mathfrak{O_i}^{e_{H_\mathfrak{O_i}}-1}\otimes\theta_\mathfrak{O_j}^{e_{H_\mathfrak{p_j}}-1}}=R/(\Delta_\mathfrak{O_i}\Delta_\mathfrak{O_j})\] for $1\leq i\neq j \leq 3$, is a direct summand of $S/(z)$ over $R/(\Delta)$.

\subsection{Representation Theory of \(G(m,1,n)\)}

Since the groups $G(m,1,n)$ are extensions of the symmetric group, their representation theory looks similar. We quickly introduce Young diagrams and describe how they can be used to build the representation theory of $G(m,1,n)$. A standard text for Young Tableau is \cite{fulton_1996}.

Consider $n \in \mathbb{N}\setminus \{0\}$. Let $\lambda$ be a partition of $n$, i.e $\lambda =(\lambda_1,...,\lambda_k)$, such that $k\leq n$, $0 < \lambda_{i+1}\leq \lambda_i$ and $\sum_{i}n_i=n$. 
A partition can also be represented as a Young diagram, which is constructed in the following way: Given a partition $\lambda$ of $n$, the Young diagram associated to $\lambda$ is a collection of left justified rows of squares called cells. Enumerate the rows from $0$ to $k-1$, top to bottom, the number of cells in row $i$ is $\lambda_{i}$. The partitions uniquely determine the Young diagram so we use the same notation $\lambda$ for the partition and the Young diagram. We call a Young diagram associated to a partition of $n$, a Young diagram of size $n$

\begin{example}
Let $n=5$ and $\lambda=(2,2,1)$ then the Young diagram is: 

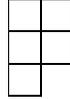
\[\yng(2,2,1)\]
\end{example}

\begin{definition}
A \textit{Young tableau} is a Young diagram of size $n$, where each cell contains a number from $1$ to $n$ such that each number $1$ to $n$ appears only once. A Young tableau is called \textit{standard} the sequence of entries in the rows and columns are strictly increasing.
\end{definition}

\begin{example}
Let $n=5$ and consider the Young diagram $\lambda$ from the previous example. Then the following are Young tableau:

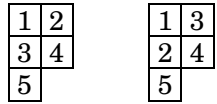
\[\young(12,34,5) \hspace{1cm} \young(13,24,5)\]

These are also both standard tableau.
\end{example}
It is widely known that Young diagrams of size $n$ are in bijection with the irreducible representation of the symmetric group on $n$ letters, $S_n$, see \cite{fulton_1996}. The representation theory of the symmetric group can be extended to the representation theory of $G(m,1,n)$ by instead considering $m$-tuples of Young diagrams.

\begin{definition}
Let $m,n \in \mathbb{Z^+}$.

\begin{enumerate}
    \item[i)] Let $\mathcal{P}_{m,n}$ the set of all $m$-tuple of Young diagrams $\lambda=(\lambda^{(0)},...,\lambda^{(m-1)})$, where $\lambda^{(i)}$ is a partition of $n_i$ for $0 \leq n_i \leq n$ for $0 \leq i \leq m-1$ and such that $\sum_{0 \leq i \leq m-1}n_i = n$.  
\end{enumerate}
Let $\lambda \in \mathcal{P}_{m,n}$.
\begin{enumerate}

    \item[ii)] An \textit{$m$-tuple of Young tableau of shape $\lambda$} is an $m$-tuple of Young diagrams with the numbers $1$ to $n$ enumerating the cells.
    \item[iii)] An $m$-tuple of Young tableau of shape $\lambda$ is called \textit{standard} if the sequence of entries in the rows and columns of each $\lambda^{(i)}$ are strictly increasing. The set of all standard $m$-tuples of Young tableau of shape $\lambda$ is denoted $\stab(\lambda)$
\end{enumerate}
\end{definition}

\begin{example}
Let $m= 3$ and $n = 5$ then $\lambda= ({\Yvcentermath1\yng(1,1)},-,{\Yvcentermath1\yng(2,1)})$ is an element of $\mathcal{P}_{m,n}$ and $\lambda_1= ({\Yvcentermath1\young(2,5)},-,{\Yvcentermath1\young(14,3)})$ is a standard $m$-tuple of Young tableau of shape $\lambda$.
\end{example}

\begin{lemma} 
Every irreducible representation of $G(m,1,n)$ corresponds to an element of $\mathcal{P}_{m,n}$. 
\end{lemma}
 
 The following is determined in \cite{ariki1994hecke}. Take a $m$-tuple of Young diagrams $\lambda$ and let $V_\lambda$ be the vector space spanned of all linear combination of all possible $m$-tuples of standard Young tableau of shape $\lambda$. Then the generators of $G(m,1,n)$ act on the $m$-tuples of Young tableaux as follows. First let $\xi$ be a $m^{th}$ root of unity. Let $\lambda$ be a $m$-tuple of Young diagrams, then let $t_p$ denote a $m$-tuple of standard Young tableaux, i.e a basis vector for the representation corresponding to $\lambda$. Let $1 \leq k\leq n-1$, assume that $k,k-1$  can be swapped and still create a $m$-tuple of standard Young tableaux. Let $t_q$ be the $m$-tuple of tableau such that $k,k-1$ are switched in $t_p$. Then recalling that the generators of $G(m,1,n)$ are denoted $s_1,...,s_n$, we define an action on $\lambda$:

\[s_1(t_p)=\xi^{i}t_p\]

when 1 appears in the $i^{th}$ position of the $m$-tuple of tableau. 

\[s_k(t_p) = 
\begin{cases}
  t_p & \text{if $k-1$ and $k$ are in the same row}  \\
  -t_p & \text{if $k-1$ and $k$ are in the same column}\\
  t_q & \text{otherwise}
\end{cases}
\]
\begin{example}\label{Example:dimlamda}

Consider $G=G(4,1,2)$, the $m$-tuple of Young diagrams $\lambda_1=({\Yvcentermath1\yng(2)},-,-,-)$ corresponds to a $1$ dimensional representation $V_{\lambda_1}$, since there is only one standard Young tableau of shape $\lambda$: $({\Yvcentermath1\young(12)},-,-,-)$. Similarly $\lambda_2 = ({\Yvcentermath1\yng(1)},{\Yvcentermath1\yng(1)},-,-)$ is a 2-dimensional representation since there are two standard Young tableau of shape $\lambda_2$: $( {\Yvcentermath1\young(1)},{\Yvcentermath1\young(2)},-,-)$ and $({\Yvcentermath1\young(2)},{\Yvcentermath1\young(1)},-,-)$
\end{example}

While we can describe the representation theory of the general groups $G(m,1,n)$ from now on focus will be on the groups $G(m,1,2)$, since their discriminants are curve singularities.

Recalling for $G(m,1,2)$ the generators are

\[s_1=\begin{bmatrix}\xi_m & 0 \\ 0 & 1 \end{bmatrix}, s_2=\begin{bmatrix}0 &1 \\ 1 & 0 \end{bmatrix}\]
are the generators for $G(m,1,2)$.

\begin{example}
Continuing from Example \ref{Example:dimlamda},
let $Q=({\Yvcentermath1\young(12)},-,-,-)$

$$ s_1(Q)=Q \hspace{1cm} s_2(Q) = Q$$

let $P =  ( {\Yvcentermath1\young(1)},{\Yvcentermath1\young(2)},-,-)$ and $ T = ({\Yvcentermath1\young(2)},{\Yvcentermath1\young(1)},-,-)$. The action of the generators $s_1$ and $s_2$ of $G(4,1,2)$ is as follows: 
\begin{align*}
s_1(P)=P & \hspace{1cm} s_2(P)= T  \\
s_1(T)=\xi_4T & \hspace{1cm} s_2(T)= P
\end{align*}

\end{example}

\begin{lemma}
Let $G=G(m,1,2)$, then $V_{\det}$ is represented by the $m$- tuple of Young tableaux:

\[\alpha=\left\{\left(-,{\Yvcentermath1 \yng(1,1)},-, \dots, - \right)\right\}\]
\end{lemma} 

\begin{proof}
We first note the representation given by $\alpha$ is one dimensional since the only possible Young tableau is: $\alpha_1=\left(-,{\Yvcentermath1 \young(1,2)},-,\dots ,- \right)$. Recall that $s_1$,$s_2$  are the generators of the group, and denote by $\xi$ the $m^{th}$ root of unity. We then have the following actions on the Young tableau $\alpha_1$: 

\[s_1\alpha_1 = \xi\alpha_1 \]
\[s_2\alpha_1=-\alpha_1\]
Thus this is the representation $V_{\det}$, since $s_1 \mapsto \xi=\det (s_1)$, $s_2 \mapsto -1 = \det(s_2)$.
\end{proof}

\section{Higher Specht Polynomials}
Let $G=G(m,1,n)$, $S=\mathbb{C}[x_1,...,x_n]$ and $R_+$ the ideal in $S$ generated by the basic invariants $f_1,...,f_n$ of $G$. This section will be about calculating bases for components of the coinvariant algebra $S/R_+$. In \cite{ariki1997} a basis was calculated via defining Higher Specht Polynomials for the groups $G(m,1,n)$ and then later was generalised for the $G(m,p,n)$ in \cite{morita1998}.

\begin{example}\label{ex}
Let $\lambda = \rep{a}{b}$,  $Q=\repg{1}{a}{2}{b}$ and $T=\repg{2}{a}{1}{b}$, then $Q,T \in \mathrm{STab}(\lambda)$.
\end{example}

Let $\lambda\in \mathcal{P}_{m,n}, Q \in \mathrm{STab}(\lambda)$, we create a word $w(Q)$ by first reading columns from bottom to top starting with the left column of $Q^{(0)}$, we then move onto $Q^{(1)}$ and so on, until we have read all the components. For a word $w(Q)$ we then define the index $i(w(Q))$ inductively as follows: the number $1$ has index $i(1)=0$, let the number $p$ have index $i(p)=k$, then if $p+1$ is to the left of $p$ then $i(p+1)=k+1$, if $p+1$ is to the right of $p$ then $i(p+1)=k$. We then assign to $i(w(Q))$ a tableau $i(Q)$ of shape $\lambda$ with the entries of the cells corresponding to their index.
\begin{example}
Following from Example \ref{ex}, $w(Q)=12$,  $w(T)=21$, $i(w(Q))=00$ and $i(w(T))=10$ so $i(Q)=\repg{0}{a}{0}{b}$ and $i(T)= \repg{1}{a}{0}{b}$.
\end{example}

\begin{definition}
    Let $T$ be an $m$-tuple of tableau with shape $\lambda$, for each component $T^{(a)}$, the \textit{Young Symmetrizer}, $e_{T^{(a)}}$ is defined by:
    
    \[e_{T^{(a)}}:=\frac{1}{\alpha_{T^{(a)}}} \sum_{\sigma \in R(T^{(a)}) \tau \in C(T^{(a)})} \text{sgn}(\tau)\tau \sigma,\]
    
    where $\alpha_{T^{(a)}}$ is the hook length of $\lambda^{(a)}$ and $R(T^{(a)}),C(T^{(a)})$ are the row and column stabilizers of $T^{(a)}$ respectively.
\end{definition}

\begin{example}
Continuing from Example \ref{ex}, $e_{Q^{(a)}}=e_{T^{(a)}}=$id since the the components of $Q$ and $T$ have at most one cell and thus the row and column stabilizers are the identities.
\end{example}

\begin{definition}
    Let $\lambda=(\lambda^{(0)},...,\lambda^{(m-1)})$ be a $m$-tuple of Young diagrams $Q\in \mathrm{STab}(\lambda)$, $T \in \mathrm{Tab}(\lambda)$, $x=(x_1,...,x_n)$. Define the Higher Specht polynomial as:
    
    \[\Delta_{Q,T}(x)= \prod^{m-1}_{a=0}\left( e_{T^{(a)}}(x^{mi(Q)^{(a)}}_{T^{(a)}}) (\prod_{k\in T^{(a)}}x^a_k)\right)\]
    
    Where
    \[x^{mi(Q)^{(a)}}_{T^{(a)}}= \prod_{C \in \lambda^{(a)}} x^{mi(Q)^{(a)}(C)}_{T^{(a)}(C)} \] 
\end{definition}

\begin{example}\label{spec}
Again using $Q,T$ from the previous examples. $x^{mi(Q)^{(a)}}_{T^{(a)}}=x_2^0$, $x^{mi(Q)^{(b)}}_{T^{(b)}}=x_1^0$  so $\Delta_{Q,T}(x)=e_{T^{(a)}}(x^0_2)(x_2^a)e_{T^{(b)}}(x^0_1)(x_1^b)= x_2^ax_1^b$, $\Delta_{Q,Q}(x)= x_1^ax_2^b$. Then $x^{mi(T)^{(a)}}_{T^{(a)}}=x_2^{m}$ and $x^{mi(T)^{(b)}}_{T^{(b)}}=x_1^{m}$ $\Delta_{T,T}(x)=x_2^{a+m}x_1^b$ and $\Delta_{T,Q}(x)=x_1^{a+m}x_2^b$.
\end{example}

\begin{theorem}\cite{ariki1997} Let $\lambda\in \mathcal{P}_{m,n}$ and $Q\in \mathrm{STab}(\lambda)$ then:
\begin{itemize}
    \item The subspace $V_Q(\lambda) = \sum_{T \in \mathrm{Tab}(\lambda)}\mathbb{C}\Delta_{Q,T}(x)$ of $Q$ affords an irreducible representation of $G(m,1,n)$
    \item The set $\{\Delta_{Q,T}:T \in \mathrm{STab}(\lambda)\}$ is a basis for $V_Q(\lambda)$
    \item The coinvariant algebra $S_G=S/R_+$ admits an irreducible decomposition:
    
    \[S_G = \bigoplus_{\lambda\in \mathcal{P}_{m,n}} \bigoplus_{Q \in \mathrm{STab}(\lambda)}(V_Q(\lambda)\text{ mod } R_+)\]
\end{itemize}
\end{theorem}

Continuing from Example \ref{spec}, $\{\Delta_{Q,T},\Delta_{Q,Q}\}$ is a basis for $V_Q(\lambda)$ and $\{\Delta_{T,T},\Delta_{T,T}\}$ is a basis for $V_T(\lambda)$ and $V_Q(\lambda)\cong V_T(\lambda)$.

For the case of $G=G(m,1,2)$ we describe the dimension two components of $S_G$ in the above decomposition.
\begin{definition}
    
Fix $G(m,1,2)$ and let $\alpha$ be a Young diagram, then let $\alpha_i=(\alpha_i^{(0)},...,\alpha_i^{(m-1)})$ be the $m$-tuple of Young diagrams such that $\alpha_i^{j}=0$ for all $0\leq j \leq m-1$ with $j\neq i$ and $\alpha_i^i=\alpha$. i.e $\alpha_i$ is the $m$-tuple of Young diagrams with $\alpha$ in the $i^{th}$ position and nothing in every other position.
\end{definition}

Extend this notation to $\alpha_i\beta_j....$ to mean the $m$-tuple of Young diagrams with $\alpha$ in position $i$, $\beta$ in position $j$ and so on.

\begin{example}
Let $G=G(4,1,2)$ then $W_{\det}= {\Yvcentermath1 \yng(1,1)}_1$, and $({\Yvcentermath1 \yng(1)},-,{\Yvcentermath1 \yng(1)},-)= {\Yvcentermath1 \yng(1)}_0{\Yvcentermath1 \yng(1)}_2$
\end{example}

The following can easily be seen:

\begin{lemma} The representations of $G(m,1,2)$ are of the following form:
\begin{itemize}
    \item The one dimensional representations corresponds to the $m$-tuple of diagrams; ${\Yvcentermath1 \yng(2)}_i$ or ${\Yvcentermath1 \yng(1,1)}_i$ for $0\leq i \leq m-1$ orientation. 
    
    \item The two dimensional representations corresponds to the $m$-tuple of diagrams $\yng(1)_i\yng(1)_j$, where $0\leq i<j \leq m-1$
\end{itemize}
\end{lemma}{}
\begin{proof}
Building representations of $G(m,1,2)$ is done by placing two cells in an $m$-tuple. There are two cases; the cells are in different positions or they are in the same position. In the first case these correspond to the two-dimensional representation. In the second case, there is the choice of the two cells being in the same column or in the same row - these correspond to the 1-dimensional representations.
\end{proof}

\begin{lemma}
Let $\alpha$ be an $m$-tuple of Young diagrams corresponding  to the representation $W_\alpha$ of $G$. The Young diagram $\beta$ representing the representation $W_\alpha\otimes W_{\det}$ is obtained by the following:
\begin{itemize}
    \item If $W_\alpha$ is one dimensional then $\alpha$ is of the form; ${\Yvcentermath1 \yng(2)}_i$ or ${\Yvcentermath1 \yng(1,1)}_i$ for some $0 \leq i <m$. Then $\beta$ is;  ${\Yvcentermath1 \yng(1,1)}_{i+1}$ or ${\Yvcentermath1 \yng(2)}_{i+1}$ respectively where $i+1$ is taken modulo $m$. 
    
    \item If $W_\alpha$ is a 2 dimensional representation, then $\alpha$ is of the form ${\Yvcentermath1 \yng(1)}_{i}{\Yvcentermath1 \yng(1)}_{j}$ where $0\leq i <j<m$, then $\beta$ is of the form ${\Yvcentermath1 \yng(1)}_{i+1}{\Yvcentermath1 \yng(1)}_{j+1}$ where $i+1, j+1$ are taken modulo $m$.
\end{itemize}{}
\end{lemma}

\begin{proof}

Let $\alpha$ be of the form ${\Yvcentermath1 \yng(2)}_i$ then $s_1(\alpha)=\xi^i\alpha$, where $\xi$ is the $m^{th}$ root of unity. $\alpha$ is a one-dimensional vector space as it spanned by the tableau $t_1={\Yvcentermath1 \young(12)}_i$, and $\alpha \otimes \det$ is generated by $t_1 \otimes {\Yvcentermath1 \young(1,2)}_1$.

\begin{equation*}
s_1(t_1 \otimes {\Yvcentermath1 \young(1,2)}_1)=s_1(t_1)\otimes s_1({\Yvcentermath1 \young(1,2)}_1) =\xi^i t_1 \otimes \xi{\Yvcentermath1 \young(1,2)}_1= \xi^{i+1}(t_1 \otimes {\Yvcentermath1 \young(1,2)}_1)  \end{equation*}

\begin{align*}
s_2(t_1 \otimes {\Yvcentermath1 \young(1,2)}_1)=&s_2(t_1)\otimes s_2({\Yvcentermath1 \young(1,2)}_1) =\begin{cases}
  t_1 \otimes -{\Yvcentermath1 \young(1,2)}_1  & \text{if $1$ and $2$ are in the same row in $t_1$}  \\
  -t_1 \otimes -{\Yvcentermath1 \young(1,2)}_1& \text{if $1$ and $2$ are in the same column in $t_1$}\\
  t_1  \otimes -{\Yvcentermath1 \young(1,2)}_1& \text{otherwise}
\end{cases}
\\ =& -t_1  \otimes {\Yvcentermath1 \young(1,2)}_1
\end{align*}
Thus the young tableau associated with $\alpha \otimes \det$ is ${\Yvcentermath1 \yng(1,1)}_{i+1}$ where $i+1$ is taken mod $m$. The case $\alpha={\Yvcentermath1 \yng(1,1)}_i$ is similar.

Let $\alpha$ be of the form ${\Yvcentermath1 \yng(1)}_{i}{\Yvcentermath1 \yng(1)}_{j}$ where $0\leq i <j \leq m-1$, $\alpha$ is then a two-dimensional representation spanned by the tableau $t_1= {\Yvcentermath1 \young(1)}_{i}{\Yvcentermath1 \young(2)}_{j}, -t_2={\Yvcentermath1 \young(2)}_{i}{\Yvcentermath1 \young(1)}_{j}$. Note that the negative sign is here for clarity 
\begin{align*}
s_1(t_1 \otimes {\Yvcentermath1 \young(1,2)}_1)=&s_1(t_1)\otimes s_1({\Yvcentermath1 \young(1,2)}_1) =\xi^i t_1 \otimes \xi{\Yvcentermath1 \young(1,2)}_1= \xi^{i+1}(t_1 \otimes {\Yvcentermath1 \young(1,2)}_1) \\ s_1(-t_2 \otimes {\Yvcentermath1 \young(1,2)}_1)=&s_1(-t_2)\otimes s_2({\Yvcentermath1 \young(1,2)}_1) =-\xi^j t_2\otimes \xi{\Yvcentermath1 \young(1,2)}_1= \xi^{j+1}(-t_2 \otimes {\Yvcentermath1 \young(1,2)}_1)   \end{align*}
\begin{align*}
s_2(t_1 \otimes {\Yvcentermath1 \young(1,2)}_1)=&s_2(t_1)\otimes s_2({\Yvcentermath1 \young(1,2)}_1) =\begin{cases}
  t_1 \otimes -{\Yvcentermath1 \young(1,2)}_1  & \text{if $1$ and $2$ are in the same row}  \\
  -t_1 \otimes -{\Yvcentermath1 \young(1,2)}_1& \text{if $1$ and $2$ are in the same column}\\
  t_2  \otimes -{\Yvcentermath1 \young(1,2)}_1& \text{otherwise}
\end{cases}
\\ =& -t_2  \otimes {\Yvcentermath1 \young(1,2)}_1
\end{align*}
\begin{align*}
s_2(-t_2 \otimes {\Yvcentermath1 \young(1,2)}_1)=&s_2(t_2)\otimes s_2({\Yvcentermath1 \young(1,2)}_1) =\begin{cases}
  -t_2 \otimes -{\Yvcentermath1 \young(1,2)}_1  & \text{if $1$ and $2$ are in the same row}  \\
  t_2 \otimes -{\Yvcentermath1 \young(1,2)}_1& \text{if $1$ and $2$ are in the same column}\\
  -t_1  \otimes -{\Yvcentermath1 \young(1,2)}_1& \text{otherwise}
\end{cases}
\\ =& t_1  \otimes {\Yvcentermath1 \young(1,2)}_1
\end{align*}
Thus is isomorphic to the representation corresponding to the tableau  ${\Yvcentermath1 \yng(1)}_{i+1}{\Yvcentermath1 \yng(1)}_{j+1}$ where $i+1, j+1$ are taken modulo $m$.
\end{proof}

\begin{lemma}\label{mybasi}
Let $G=G(m,1,2)$, and let $\lambda={\Yvcentermath1 \yng(1)}_{i}{\Yvcentermath1 \yng(1)}_{j}$. Then using Higher Specht polynomials, a basis for the isotypical component of $S_G$ isomorphic to $W_\lambda$ is $\{x^{i}y^{j},x^{j}y^{i},x^{i+m}y^{j},x^{j}y^{i+m}\}$.
\end{lemma}

\begin{proof}
See Example \ref{spec}.
\end{proof}

From now on if $\lambda$ is a $m$-tuple of Young diagrams, we will also denote the representation $V_{\lambda}$ corresponding to it as $\lambda$.
 \begin{lemma}\label{Lastrep}
Let $G=G(m,1,2)$ and $\lambda={\Yvcentermath1 \yng(1)}_{m-2}{\Yvcentermath1 \yng(1)}_{m-1}$ the map of $z|_{\lambda}$ on the isotypical component of type $\lambda$ is matrix factorization equivalent to:
    
    \[\begin{tikzcd}[ampersand replacement=\&, column sep =2.5cm]
S^G_{\lambda \otimes \det} \arrow{r}{j|_{\lambda \otimes \mathrm{det}}} \& S^G_{\lambda} \arrow{r}{ \begin{bmatrix}
-2\sigma_1 & \sigma_2\sigma_1 \\ \sigma_2 & -2\sigma_1
    \end{bmatrix} \otimes I_2}  \& S^G_{\lambda \otimes \det}\end{tikzcd}\]

\end{lemma}
\begin{proof}
From the above, a basis for the isotypical component of $S/R_+$ corresponding to $\lambda$ is:
\[\{x^{m-2}y^{m-1},x^{m-1}y^{m-2},x^{2m-2}y^{m-1},x^{m-1}y^{2m-2}\}.\]

And for a basis for the isotypical component corresponding to $\lambda \otimes \textrm{det}$ is: 
\[\{x^{m-1},y^{m-1},x^{m-1}y^m,x^{m}y^{m-1}\}\]

Then, multiplication of $z$ on the basis elements can be calculated as follows:
\begin{equation*}
\begin{split}
    z(x^{m-2}y^{m-1})&=x^{m-1}y^m(x^m-y^m) \\
    &=x^{2m-1}y^m-x^{m-1}y^{2m} \\
    &=-\sigma_2(x^{m-1}y^m)+2\sigma_1x^{m-1} \\
    z(x^{m-1}y^{m-2})&=\sigma_2(x^{m}y^{m-1})-2\sigma_1y^{m-1}
\end{split}
\end{equation*}

and
\begin{equation*}
    \begin{split}
        z(x^{2m-2}y^{m-1})&= x^{2m-1}y^m(x^m-y^m) \\ &= \sigma_2\sigma_1(x^{m-1})-2\sigma_1(x^{m-1}y^m)
    \end{split}
\end{equation*}

We obtain the matrix: 
    \[\begin{bmatrix}
    -2\sigma_1 & 0 & \sigma_2\sigma_1 & 0 \\ 0 & -2\sigma_1 & 0 & \sigma_2\sigma_1 \\ \sigma_2 & 0 & -2\sigma_1 & 0 \\ 0 & \sigma_2 & 0 & -2\sigma_1
    \end{bmatrix}= \begin{bmatrix} -2\sigma_1 & \sigma_2\sigma_1 \\ \sigma_2 & -2\sigma_1\end{bmatrix} \otimes I_2\]

\end{proof}

Matrix factorizations were used in \cite{YujiCM} Chapter 9 to calculate the $CM$ modules over one dimensional ADE singularities. Note that Yoshino calculated them for the completed case, but from Remark \ref{grade} we can instead consider them as graded matrix factorizations over the graded local ring $S$.

\begin{definition}
  Let $\Lambda=\mathbb{C}[x,y]/(f)$, where $f=x^2+y^{4}$, then Spec$(\Lambda)$ is an $A_3$ singularity.
\end{definition}

\begin{remark}
 Let $G=G(m,1,2)$ and recall that $\Delta=\sigma_1(\sigma_2^2-4\sigma_1)$. Then $\spec(R/(\Delta))$ is an $A_3$-singularity. This is seen by the coordinate change $x= \sigma_1- \frac{\sigma_2^2}{8}$ and $y = \frac{\sigma_2}{2\sqrt{2}}$.
\end{remark}

\begin{theorem}\cite{YujiCM} (9.9)
The non-trivial indecomposable $CM$ modules over $\mathbb{C}[\sigma_1,\sigma_2]/(\Delta)$, where $\Delta=\sigma_1(\sigma_2^2-4\sigma_1)$ are given as cokernels of the matrix factorizations: $(\alpha,\beta)=(\sigma_1,\sigma_2^2-4\sigma_1)$, $(\beta,\alpha)=(\sigma_2^2-4\sigma_1,\sigma_1)$ and the matrix:

\[(\phi,\psi)=(\begin{bmatrix} 2\sigma_1 & \sigma_2\sigma_1 \\ \sigma_2 & 2\sigma_1
\end{bmatrix},\begin{bmatrix} -2\sigma_1 & \sigma_2\sigma_1 \\ \sigma_2 & -2\sigma_1
\end{bmatrix})\]

Denote the modules: $A = \cok(\alpha,\beta)$, $B = \cok(\beta,\alpha)$, $X=\cok(\phi,\psi)$. Note $\cok(\phi,\psi)=\cok(\psi,\phi)$. These are the indecomposable modules in $CM(\Lambda)$.
\end{theorem}

Thus the matrix factorization from Lemma \ref{Lastrep} is equivalent (as matrix factorizations) to $\phi\oplus\phi$ (and also $\psi\oplus \psi$), and so  $M_{\rep{m-2}{m-1}}\cong X \oplus X.$
\begin{remark}
The graded local case, the category $CM(\Lambda)$ is Krull-Schmidt and so direct sum decomposition are unique. We take the grading of the complete modules and see that the above list is the complete set of non-isomorphic up up to degree shifting indecomposable objects in $CM(R/(\Delta))$.
\end{remark}

We need to check that the support of $S/(z)$ on $R/(\Delta)$ modules is the entire list of isomorphism classes of modules in $CM(R/(\Delta))$.

\begin{theorem}\label{first}
Let $G=G(m,1,2)$ then the following is a decomposition of $S/(z)$ as $R/(\Delta)$ $CM$ modules:

\begin{equation*}
    \begin{split}
        S/(z) &\cong R/(\sigma_1)\oplus R/(\Delta) \oplus (R/(\sigma_2^2-4\sigma_1))^{2{{m-1}\choose 2}+m-1} \oplus X^{2m-2} \\
        & = A\oplus B^{2{{m-1}\choose 2}+m-1} \oplus X^{2m-2}.
    \end{split}
\end{equation*}

In particular $S/(z)$ is a representation generator for $CM(R/(\Delta))$.

\end{theorem}

\begin{proof}
From Section \ref{lin}, the linear characters give the modules $M_{\theta_\mathfrak{O}^{m-1}}=R/(\Delta)_{\mathfrak{O}}=R/(\sigma_1)$ and $M_{\theta_\mathfrak{q}}=R/(\Delta)_{\mathfrak{q}}=R/(\sigma_2^2 -4\sigma_1)$. Moreover for the modules  $M_{\theta_\mathfrak{O}^i \otimes\theta_\mathfrak{q} }\cong R/(\sigma_2^2 -4\sigma_1)$ for $1\leq i \leq m-2$. From calculations similar to those above, the 2 dimensional characters ${\Yvcentermath1 \yng(1)}_{i}{\Yvcentermath1 \yng(1)}_{j}$ for $0 \leq i < j <m-1$ each give a copy of $(R/(4\sigma_1-\sigma_2^2))^2$. There are $m-1 \choose 2$ of them. Then $M_{\rep{i}{m-1}} \cong X^2$ for $0 \leq i < m-1$. 
\end{proof}
\begin{corollary}
Let $G=G(m,1,2)$ then $\text{End}_{R/(\Delta)}(S/(z))$ is an NCR for $R/(\Delta)$. 
\end{corollary}

\section{Irreducible representations of \(G(m,p,n)\)}
The irreducible representations of $G(m,p,n)$ were first described by Stembridge, see \cite{stembridge1989}. Let $G=G(m,1,n)$ and $H=G(m,p,n)$, $q:=\frac{m}{p}$, we define a linear character of $G$ by:

\[\delta_i=(-,\dots,{\Yvcentermath1\yng(1)\cdots \yng(1)},\dots,-)\]

It is straightforward to check that the action $-\otimes \delta_i$ "shifts" the diagrams in a $m$-tuple of young diagrams $i$ places to the right, i.e if $\lambda=(\lambda^{(0)},\dots, \lambda^{(m-1)})$ then $\lambda\otimes \delta_i=(\lambda^{(0-i)},\dots, \lambda^{(m-1-i)})$ where ${j-i}$ is considered mod $m$. The quotient group $G/H$ is isomorphic to the cyclic group $C=\langle \delta_1^q \rangle$. Through this $G/H$ acts on the irreducible representations of $G$. Denote by $[\lambda]$ the $G/H$-orbit of the irreducible $G$-representation $\lambda$. Define an equivalence relation $\sim_H$ on Irr$(G)$ such that $\lambda\sim_H \mu$ if and only if $\lambda = \mu \otimes \delta$ for some $\delta$ in $G/H$.

We define some numerology, let $b(\lambda)$ be the cardinality of the $G/H$ orbit of $\lambda$ and $u(\lambda):=\frac{p}{b(\lambda)}$. Define the stabilizer of $\lambda$ as the subgroup of $G/H$:

\[(G/H)_\lambda:=\{\delta \in G/H : \lambda \otimes \delta = \lambda\}\]

This is then generated by $\delta_1^{b(\lambda)\cdot q}$ with an order of $u(\lambda)$. 

\begin{example}\label{stabex}
Let $G=G(4,1,2)$ and $H=G(4,2,2)$. Let $\lambda_1=({\Yvcentermath1\yng(1)},{\Yvcentermath1\yng(1)},-,-),$ $\lambda_2=({\Yvcentermath1\yng(1)},-,{\Yvcentermath1\yng(1)},-)$ $\lambda_3=(-,{\Yvcentermath1\yng(1)},-,{\Yvcentermath1\yng(1)})$ and $\lambda_4=(-,-,{\Yvcentermath1\yng(1)},{\Yvcentermath1\yng(1)})$. The quoitient group $G/H$ is generated by $\delta_1^2$, and for example $\lambda_1\otimes \delta_1^2=\lambda_4$. Then $(G/H)_{\lambda_1}=\{1\}$, $(G/H)_{\lambda_2}=\{1,\delta_1^2\}$. $ [\lambda_1]=\{\lambda_1,\lambda_4\}$, $[\lambda_2]=[\lambda_2]$ and  $[\lambda_3=\lambda_3]$
\end{example}

\begin{theorem}(Stembridge \cite{stembridge1989})
There is a one-one correspondence between the irreducible representations $H=G(m,p,n)$ and ordered pairs $([\lambda],\delta)$ where $[\lambda]$ is the orbit of an irreducible representation $\lambda$ of $G$ and $\delta\in (G/H)_\lambda$. Then the following hold:
\begin{itemize}
    \item[a)] $\text{Res}^G_H(\lambda)=\text{Res}^G_H(\mu)$ for $\lambda\sim_H\mu$.
    \item[b)] $\text{Res}^G_H(\lambda)=\bigoplus_{\delta\in(G/H)_\lambda}([\lambda],\delta)$.
    \item[c)] $\text{Ind}^G_H([\lambda],\delta)=\bigoplus_{\mu \text{ s.t } \mu \sim_H \lambda} \mu$.
\end{itemize}
\end{theorem}
\begin{example}
Following from Example \ref{stabex}, $\mathrm{Res}^G_H\lambda_1=[\lambda_1]=\mathrm{Res}^G_H\lambda_4$ and $\mathrm{Res}^G_H\lambda_2=([\lambda_2],1) \oplus ([\lambda_2],\delta_1^2)$
\end{example}

We can fully describe the irreducible representations of $G(m,p,2)$ in the following way. Let $G=G(m,1,2)$ and $H=G(m,p,2)$ where $q=\frac{m}{p}$. Recall that the representations of $G(m,1,2)$ are of the following form:

   \begin{itemize}
    \item The one dimensional representations corresponds to the $m$-tuple of diagrams; ${\Yvcentermath1 \yng(2)}_i$ or ${\Yvcentermath1 \yng(1,1)}_i$ for $0\leq i \leq m-1$. 
    
    \item The two dimensional representations corresponds to the $m$-tuple of diagrams; $\yng(1)_i\yng(1)_j$, where $0\leq i<j \leq m-1$.
\end{itemize}
The one-dimensional representations will again stay irreducible under restriction. The restriction of a two dimensional representation $\lambda$ is reducible if $(G/H)_\lambda$ is non-trivial, that means that there exists a $\delta \in G/H$, i.e $\delta=\delta_1^{q\cdot c}$ for some $c \in \{0,..,p-1\}$, such that $ \lambda\otimes\delta = \lambda$. Let $\lambda = \yng(1)_i\yng(1)_j$ then $  \lambda\otimes\delta_1^{q\cdot c}=\yng(1)_{cq+i}\yng(1)_{cq+j}$, then for $\text{Res}^G_H\lambda$ to be reducible $cq+i=j$ and $cq+j=m+i$. Thus $2cq=m$ and $cq =\frac{m}{2}$, also since $q=\frac{m}{p}$, $c\frac{m}{p}=\frac{m}{2}, 2c=p$ thus $2|p$. 

From this if $m$ is odd, or even with $p$ odd, all of the two dimensional representation stay irreducible. If $m$ is even and $p$ is even the only two-dimensional representations that will not be irreducible are $\yng(1)_i\yng(1)_{\frac{m}{2}+i}$ for $0\leq i < \frac{m}{2}$.

\begin{remark}
In the preceding example, we saw for $G=G(m,1,2)$ and $H=G(m,p,2)$ with $q=\frac{m}{p}$ that only when $p$ is even there are then two-dimensional representations that will split under restriction. These were given by $\yng(1)_i\yng(1)_{\frac{m}{2}+i}$ for $0\leq i < \frac{m}{2}$. 
\end{remark}
\subsection{Higher Specht Polynomials for \(G(m,p,2)\)}
Let $G=G(m,1,n)$ and $H=G(m,p,n)$, where $q=\frac{m}{p}$, Let $S_H$ be the coinvariant algebra for $H$. Let the operation $- \otimes \delta_1$ on the $m$-tuples of tableau defined above, be denoted by $sh$, for convenience denote by $Sh$ the operation $sh^q$. Note that $b(\lambda)$ is the smallest integer $j$ such that $Sh^{j}(\lambda)=\lambda$. For $h=1,...,m$ define

\[\mathrm{STab}(\lambda)_h=\{T=(T^{(0)},...,T^{(m-1)}) \in \mathrm{STab}(\lambda): 1 \in T^{(v)} \text{ for } 0 \leq v < h  \}\]

\begin{remark}
If $T\in \mathrm{STab}_q(\lambda)$ then $T,Sh^{b(\lambda)}(T),...,Sh^{(u(\lambda-1)b(\lambda)}(T)$ are all distinct 
\end{remark}
\begin{definition}
Let $\lambda=(\lambda^{(0)},...,\lambda^{(r-1)}) \in \mathcal{P}_{n,r}$, let $Q \in \mathrm{STab}(\lambda)$ and for $0\leq l \leq u(\lambda)-1$:

\[\Delta^{(l)}_{Q,T}(x):= \sum_{m=0}^{u(\lambda)-1}\xi^{lmdb(\lambda)}\Delta_{Q,Sh^{mb(\lambda)}(T)}(x) \]
\end{definition} 

\begin{lemma}\cite{morita1998}
Let $Q,T$ be standard $m$-tableau of shape $\lambda$ and let $l = 0,...,u(\lambda)-1$. Then the polynomial $\Delta_{Q,T}(x)$ is non-zero in $S_H$ if and only $Q\in \mathrm{STab}(\lambda)_q$.
\end{lemma}

\begin{example}
Let $\lambda=\rep{i}{j}$, where $0\leq i <j < q$ and $Q=\repg{1}{i}{2}{j}, T= \repg{2}{i}{1}{j}$, then $u(\lambda)=1$ since the only case with $u(\lambda)\neq 1$ is when $j=\frac{m}{2}+i$. Moreover this means $\Delta_{Q,T}^{(0)}= \Delta_{Q,T}=x_2^ix_2^j$ , $\Delta_{Q,Q}^{(0)}= \Delta_{Q,Q}=x_1^ix_2^j$, $\Delta_{T,Q}^{(0)}= \Delta_{T,Q}=x_1^{i+m}x_2^j$ and $\Delta_{T,T}^{(0)}= \Delta_{T,T}=x_2^{i+m}x_1^{j}$. Then since both $Q,T$ are in STab$(\lambda)_q$, all these Specht polynomials are non-zero in $S_H$. $V_Q(x) = \mathbb{C}\Delta_{Q,T}\oplus \mathbb{C}\Delta_{Q,Q}$ and $V_T(x) = \mathbb{C}\Delta_{T,Q}\oplus \mathbb{C}\Delta_{T,T}$. 
\end{example}
\begin{example}\label{diff}
Let $\lambda=\rep{i}{j}$, where $0\leq i < q \leq j \leq m-1$ with $\mathrm{Res}^G_H\lambda$ irreducible, i.e $j\neq i+\frac{m}{2}$ and $Q=\repg{1}{i}{2}{j}, T= \repg{2}{i}{1}{j}$, then $u(\lambda)=1$. Moreover this means $\Delta_{Q,T}^{(0)}= \Delta_{Q,T}=x_2^ix_2^j$ , $\Delta_{Q,Q}^{(0)}= \Delta_{Q,Q}=x_1^ix_2^j$, $\Delta_{T,Q}^{(0)}= \Delta_{T,Q}=x_1^{i+m}x_2^j$ and $\Delta_{T,T}^{(0)}= \Delta_{T,T}=x_2^{i+m}x_1^{j}$. Then since both $Q$ is in STab$(\lambda)_q$, $\Delta_{Q,T}$ and $\Delta_{Q,Q}$ are non-zero in $S_G$. Now $T$ is not in STab$(\lambda)_q$ and thus is zero in $S_G$. Let $j = cq+r$ then consider $\mu=\rep{r}{m+i-cq}$. Let $P=\repg{1}{r}{2}{m+i-cq}$ and $Q=\repg{2}{r}{1}{m+i-cq}$. Now $P\in \mathrm{STab}(\mu)_q$ and $Q \not\in \mathrm{STab}(\mu)_q$ thus the Specht polynomials $\Delta_{P,P},\Delta_{P,Q}$ are non-zero while $\Delta_{Q,P},\Delta_{Q,Q}$ are zero in $S_G$.
\end{example}
\begin{theorem}[Morita, Yamada \cite{morita1998}]\label{Spmnp} Let $\lambda=(\lambda^{(0)},...,\lambda^{(r-1)}) \in \mathcal{P}_{n,r}$. For each $S \in Stab(\lambda)$ and $0\leq l \leq u(\lambda)-1$, put $V^{(l)}_S(x)=\bigoplus_{T \in Stab(\lambda)} \mathbb{C}\Delta^{(l)}_{S,T}(x)$. Then:
\begin{itemize}
    \item The space $V_S(x)$ decomposes as $V_S(x)= \bigoplus^{u(\lambda)-1}_{l=0} V^{(l)}_S(x)$.
    \item The space $V^{(l)}_S(\lambda)$ affords an irreducible representation of $G(m,r,n)$.
    \item The $G(m,p,n)$-module $S_H$ admits an irreducible decomposition:
    
    \[S_H= \bigoplus_{\lambda} \bigoplus_{S \in \mathrm{STab}(\lambda)_d} \bigoplus_{l=0}^{u(\lambda)-1} V_S^{(l)}\]
 
\end{itemize}
\end{theorem}

 \begin{example}
 Continuing Example \ref{diff},
the vector space $V_P(x)$ then stays irreducible and is isomorphic to $V_S(x)$ since $\lambda \sim_H \mu$. Thus a basis for the isotypical component $S^H_\lambda$ of $S_H$ of type $\lambda$ is:

\[\{\Delta^{(0)}_{S,S},\Delta^{(0)}_{S,T}, \Delta^{(0)}_{Q,P},\Delta^{(0)}_{Q,Q}\}\]
 \end{example}

\section{\(S/(z)\) for \(G(m,p,2)\)}\label{Mainse}

As before, let $G=G(m,1,2)$, $H=G(m,p,2)$ with $m = pq$, $S=\mathbb{C}[x,y]$, $R=S^H \cong \mathbb{C}[\sigma_1,\sigma_2]$ where $\sigma_1=x^m+y^m,$ $\sigma_2=(xy)^q$. For the case $p \neq m$, recall that $z=xy(x^m-y^m)$ and $\Delta=\sigma_2(\sigma_1^2-4\sigma_1^p)$.

\begin{lemma}
The representation $\det_H=\text{Res}^G_H\det_G$.
\end{lemma}

\begin{proof}
Obvious, since $\det$ is one dimensional and for $h \in H$, $\mathrm{Res}^G_H\det(h)=\det(h)$.
\end{proof}

\begin{lemma}
Let $0\leq i < j< m$ be such that $\mathrm{Res}^G_H \rep{i}{j}$ is irreducible then $\mathrm{Res}^G_H \rep{i}{j}\otimes \det_H=\mathrm{Res}^G_H\rep{i+1}{j+1}$ where $i+1$, $j+1$ are taken mod $m$ and is irreducible.
\end{lemma}

\begin{proof}
$\mathrm{Res}^G_H \rep{i}{j}\otimes \det_H =\mathrm{Res}^G_H \rep{i}{j}\otimes \mathrm{Res}^G_H\det_G =\mathrm{Res}^G_H( \rep{i}{j}\otimes \det_G)  =\mathrm{Res}^G_H \rep{i+1}{j+1}$
\end{proof}

Bases for the Isotypical components of $S/(R_+)$ of type $\mathrm{Res}^G_H \rep{i}{j}$ can again be calculated as Higher Specht Polynomials see \cite{morita1998}. The goal is to show that $S/(z)$ is a representation generator for $CM(R/(\Delta))$. A full list of isomorphism classes of Cohen-Macaulay modules can be found in \cite{YujiCM}. The rest of this section is dedicated to showing that there is at least one of each module in that list. 

We will breakdown the general case into 3 distinct cases:\begin{enumerate}
    \item[(1)] $m\neq p$ and $p$ odd,
    \item[(2)] $m\neq p$ and $p$ even,
    \item[(3)] $m = p$.
\end{enumerate}

\subsection{ \(m\neq p \) and \(p\) odd.} Recall for the groups $G(m,p,2)$ with $q=\frac{m}{p}$ the invariants are $\sigma_1 =(xy)^q$ and $\sigma_2=x^m-y^m$. The discriminant $\Delta=\sigma_1(\sigma_2^2-4\sigma_1^p)$ of $H$ defines a $D_{p+2}$ singularity. When $p$ is odd there are $2$ matrix factorizations of the form $(\sigma_2^2-4\sigma_1^{p},\sigma_1)$, $(\sigma_1,\sigma_2^2-4\sigma_1^{p})$ and thus the two modules $A=\cok(\sigma_1,\sigma_2^2-4\sigma_1^{p})$ and  $B=\cok(\sigma_2^2-4\sigma_1^{p},\sigma_1)$. Using the results from Section \ref{lin}, we can construct one dimensional representations $\lambda_1$,$\lambda_2$ of $G(m,p,2)$ such that $(z|_{\lambda_1},j|_{\lambda_1\otimes{\det}})$ is equivalent to $(\sigma_1,\sigma_2^2-4\sigma_1^{p})$ and $(z|_{\lambda_2},j|_{\lambda_2\otimes{\det}})$ is equivalent to $(\sigma_2^2-4\sigma_1^{p},\sigma_1)$. 

For $R/(\Delta)$ there are four more families of indecomposable $CM$-modules, $X_j,Y_j,K_j,N_j$, where $0\leq j \leq p-1$. Here we note that in \cite{YujiCM} the notation for $K_j$ is $M_j$, the change is due to clashing notation.
These modules are the cokernels of the matrix factorizations $(\phi_j,\psi_j)$ and $(\xi_j,\eta_j)$ for $0\leq j \leq p-1$, where the matrices are:

\[\phi_j= \begin{pmatrix} \sigma_2 & 2\sigma_1^j \\ 2\sigma_1^{p-j} & \sigma_2 \end{pmatrix}\hspace{1cm} \psi_j = \begin{pmatrix} \sigma_2\sigma_1 & -2\sigma_1^{j+1} \\ -2\sigma_1^{p+1-j} & \sigma_2\sigma_1 \end{pmatrix} \]

\[\xi_j=\begin{pmatrix}\sigma_2 & 2\sigma_1^j \\ 2\sigma_1^{p+1-j} & \sigma_2\sigma_1 \end{pmatrix} \hspace{1cm}  \eta_j =\begin{pmatrix}\sigma_2\sigma_1 & -2\sigma_1^j \\ -2\sigma_1^{p+1-j} & \sigma_2  \end{pmatrix}\]

For $0 \leq j \leq p-1$, we define  $X_j=\mathrm{Coker}(\xi_j,\eta_j)$, $Y_j=\mathrm{Coker}(\eta_j,\xi_j)$, $K_j=\mathrm{Coker}(\phi_j,\psi_j)$ and $N_j=\mathrm{Coker}(\psi_j,\phi_j)$. Then one can quickly show the following: 

$$X_0 \cong R/(\Delta) \cong Y_0, \hspace{1cm} K_0 \cong B, \hspace{1cm} N_0 \cong A \oplus R/(\Delta) $$
and for $1\leq j \leq p-1$;

$$X_j\cong Y_{p+1-j},\hspace{1cm} Y_j\cong X_{p+1-j}, \hspace{1cm} X_{\frac{p+1}{2}}\cong Y_{\frac{p+1}{2}},\hspace{1cm} N_j\cong N_{p-j}, \hspace{1cm} K_j\cong K_{p-j}$$.

Using these isomorphisms, and noting that every $Y$ module is isomorphic to an $X$ module, a complete list of non-isomorphic $CM$ modules over $R/(\Delta)$ is given by:
\begin{itemize}
    \item[i)] $X_j$ for $1\leq j \leq p$.
    \item[ii)] $K_j$ for $1 \leq j \leq \frac{p-1}{2}$,
    \item[iii)] $N_j$ for $1 \leq j \leq \frac{p-1}{2}$,
    \item[iv)] $A$,
    \item[v)] $B$,
    \item[vi)] $R/(\Delta)$.
\end{itemize}

\begin{theorem}\label{mod odd}
Let $H=G(m,p,2)$, with $p$ odd, $m\neq p$ and $q=\frac{m}{p} $.  The 2-dimensional representations of $H$ correspond to the following $CM$-modules in the decomposition of $S/(z)$:

\begin{itemize}
    \item[1)] $M_{\mathrm{Res}^G_H\rep{i}{jq-1}}\cong X_{p+1-j}^2$ for $0 \leq i < q-1$, $1 \leq j \leq p$.
    \item[2)] $M_{\mathrm{Res}^G_H\rep{i}{jq+r}}\cong K_{p-j}^2\cong K_j^2$ for $0 \leq j \leq p-1$, $0 \leq i < q-1$ and $0 \leq r < q-1$.
    \item[3)] $M_{\mathrm{Res}^G_H\rep{q-1}{(j+1)q-1}} \cong N_{p-j}^2 \cong N_{j}^2$ for $1 \leq j \leq p-1$.
\end{itemize}
\end{theorem}

\begin{proof}
This is proved by calculations of the following form:
\begin{itemize}
\item[1)]The module $M_{ \mathrm{Res}^G_H \rep{i}{jq-1}}$ is $X_{p+1-j}^2$ for $1 \leq j\leq p$.
\end{itemize}

Using Higher Specht polynomials, a basis for $S^G_{\mathrm{Res}^G_H \rep{i}{jq-1}}$ where $0\leq i <q-1$ and $1 \leq j\leq p$ is:

\[\{x^{jq-1}y^i,x^iy^{jq-1}, x^{q-1}y^{m-(j-1)q+i}, x^{m-(j-1)q+i}y^{q-1}\}\] 

and a basis for $S^G_{\mathrm{Res}\rep{i+1}{jq}}$ is:
\[\{x^{jq}y^{i+1},x^{i+1}y^{jq},x^{jq}y,y^{m+i+1-jq},x^{m+i+1-jq}\}\]

and so by calculation we can express multiplication by $z= xy(x^m-y^m)$ as the matrix:
\begin{equation*}
    \begin{split}
    z(x^{jq-1}y^i) &=x^{jq}y^{i+1}(x^m-y^m)\\
    &= \sigma_2(x^{jq}y^{i+1})-2\sigma_2^j(y^{m+i+1-jq}) \\
    z(x^iy^{jq-1}) &=x^{i+1}y^{jq}(x^m-y^m)\\
    &= -\sigma_2(x^{i+1}y^{jq})+2\sigma_2^j(x^{m+i+1-jq}) \\
    z(x^{q-1}y^{m-(j-1)q+i})&=x^qy^{m-(j-1)q+i+1}(x^m-y^m)\\
    &= 2\sigma_1^{p-(j-1)}(x^{jq}y^{i+1})-\sigma_1\sigma_2(y^{m-j+i+1})\\
    z(x^{m-(j-1)q+i}y^{q-1})& =x^{m-(j-1)q+i+1}y^{q}(x^m-y^m)\\
    &= -2\sigma_1^{p-(j-1)}(x^{i+1}y^{jq})-\sigma_1\sigma_2(x^{m-j+i+1})
    \end{split}
\end{equation*}
Which yields the matrix:
\[\begin{pmatrix}
 \sigma_2& 0 & 2\sigma_1^{p-(j-1)} & 0 \\ 0&-\sigma_2&0&-2\sigma_1^{p-(j-1)} \\ -2\sigma_1^j&0&-\sigma_1\sigma_2 & 0 \\ 0&2\sigma_1^j &0&\sigma_1\sigma_2
\end{pmatrix}\]

This is then equivalent as matrix factorizations of $\Delta$ to:

\[\begin{tikzcd}[ampersand replacement=\&, column sep =2.5cm]
 S^G_{\mathrm{Res}^G_H \rep{0}{jq-1}} \arrow{r}{}  \&S^G_{\mathrm{Res}^G_H \rep{1}{jq}} \arrow{r}{ \begin{pmatrix}\sigma_2 & 2\sigma_1^{(p+1-j)} \\ 2\sigma_1^{j} & \sigma_2\sigma_1 \end{pmatrix} \otimes I_2} \& S^G_{\mathrm{Res}^G_H \rep{0}{jq-1}}
\end{tikzcd}\]

Thus the module $M_{ \mathrm{Res}^G_H \rep{i}{jq-1} }\cong  X_{p+1-j}^2$ for $1\leq j \leq p$ and $0 \leq i < q-1$.

\begin{itemize}
    \item[2)] The modules $M_{\mathrm{Res}^G_H\rep{i}{jq+r}}\cong K_j^2$ for $1 \leq j \leq p-1$, $0\leq i < q-1$ and $0\leq r <0$.
\end{itemize}

A basis for
$S^G_{\mathrm{Res}^G_H \rep{0}{jq}}$ is:

\[\{x^{jq+r}y^{i},x^{i}y^{jq+r},x^{r}y^{m-jq+i},x^{m-jq+i}y^{r}\}.\]

and a basis for $S^G_{\mathrm{Res}\rep{1}{jq+1}}$ is:
\[\{x^{jq+r+1}y^{i+1},x^{i+1}y^{jq+r+1},x^{r+1}y^{m-jq+i+1},x^{m-jq+i+1}y^{r+1}\}.\]

and so by calculation we can express multiplication by $z$ as the matrix:
\[\begin{pmatrix}
 \sigma_2 & 0 & 2\sigma_1^{p-j}&0 \\ 0&-\sigma_2&0&-2\sigma_1^{p-j} \\ -2\sigma_1^{j}&0&-\sigma_2&0 \\ 0&2\sigma_1^{j} &0&\sigma_2
\end{pmatrix}\]

This is then equivalent to the matrix factorization of $\Delta$:

\[\begin{tikzcd}[ampersand replacement=\&, column sep =2.5cm]
S^G_{\mathrm{Res}^G_H \rep{0}{jq}} \arrow{r}{} \& S^G_{\mathrm{Res}^G_H \rep{1}{jq+1}} \arrow{r}{ \begin{pmatrix}\sigma_2 & 2\sigma_1^{(p-j)} \\ 2\sigma_1^{j} & \sigma_2\end{pmatrix} \otimes I_2} \& S^G_{\mathrm{Res}^G_H \rep{0}{jq}}
\end{tikzcd}\]

Thus for $1 \leq j \leq p-1$, $M_{\mathrm{Res}^G_H \rep{0}{jq}}\cong K_{p-j}^2$.

\begin{itemize}
    \item[3)]
$M_{\mathrm{Res}^G_H\rep{q-1}{(j+1)q-1}} \cong N_{p-j}^2$ for$1 \leq j \leq p-1$.

\end{itemize}

 A basis for
$S^G_{\mathrm{Res}^G_H \rep{q-1}{(j+1)q-1}}$ is:

\begin{equation*}
    \begin{split}
        \{x^{q-1}y^{(j+1)q-1},x^{(j+1)q-1}y^{q-1}
        x^{m+q-1-jq}y^{q-1},x^{q-1}y^{m+q-1-jq} \}
        \end{split}
\end{equation*}

 and a basis for $S^G_{\mathrm{Res}\rep{0}{jq}}$ is: 
\[\{x^{jq},x^{jq},y^{m-jq},y^{m-jq}\}\]

and so by calculation we can express multiplication by $z$ as the matrix:

\[\begin{pmatrix}
\sigma_1\sigma_2 & 0 & 2\sigma_1^{p-(j-1)} &0\\0& -\sigma_1\sigma_2&0&-2\sigma_1^{p-(j-1)} \\ -2\sigma_1^{j+1}&0&-\sigma_1\sigma_2 &0\\ 0&2\sigma_1^{j+1} &0&\sigma_1\sigma_2
\end{pmatrix}\]

This is then equivalent to the matrix factorization:

\[\begin{tikzcd}[ampersand replacement=\&, column sep =2.5cm]
 S^G_{\mathrm{Res}^G_H \rep{q-1}{(j+1)q-1}} \arrow{r}{} \& S^G_{\mathrm{Res}^G_H \rep{0}{jq}}\arrow{r}{ \begin{pmatrix}\sigma_2\sigma_1 & 2\sigma_1^{p-(j-1)} \\ 2\sigma_1^{j+1} & \sigma_2\sigma_1\end{pmatrix} \otimes I_2} \& S^G_{\mathrm{Res}^G_H \rep{q-1}{(j+1)q-1}}
\end{tikzcd}\]
Thus for $1 \leq j \leq p-1$, $M_{\mathrm{Res}^G_H \rep{q-1}{(j+1)q-1}}\cong N_{p-j}^2$.
\end{proof}

We fully decompose the module $S/(z);$
\begin{theorem}
Let $H=G(m,p,2)$, with $p$ odd, $m\neq p$ and $q=\frac{m}{p} $ Then:
\begin{equation*}
\begin{split}
S/(z) & \cong \bigoplus_{j=1}^{p}X^{2(q-1)}_j \oplus \bigoplus_{j=1}^{\frac{p-1}{2}} N_j^2 \oplus \bigoplus_{j=1}^{\frac{p-1}{2}}K_j^{2(q-1)^{2}} \oplus K_0^{ 2 {q-1 \choose 2}} \oplus R/(\sigma_1)\oplus (R/(\sigma_2^2 - 4\sigma_1^p))^{q-1}   \oplus R/(\Delta) \\ &\cong  \bigoplus_{j=1}^{p}X^{2(q-1)}_j \oplus \bigoplus_{j=1}^{\frac{p-1}{2}} N_j^2 \oplus \bigoplus_{j=1}^{\frac{p-1}{2}}K_j^{2(q-1)^{2}} \oplus  R/(\sigma_1)\oplus (R/(\sigma_2^2 - 4\sigma_1^p))^{q-1+2 {q-1\choose 2}}   \oplus R/(\Delta)
\end{split} 
\end{equation*}
\end{theorem}
\begin{proof}
Let $\mathfrak{O}$ be the orbit of hyperplanes that contain $\ker(x)$ and $\ker(y)$ and $\mathfrak{q}$ the other orbit. The linear characters are given by $\theta_{\mathfrak{O}}^{i}\otimes\theta_{\mathfrak{q}}^{j}$ for $0\leq i \leq q-1$ and $0\leq j \leq 1$. The modules $M_{\theta_{\mathfrak{O}}^{i}}$ for $0\leq i <q-1$ give the trivial module. The module $M_{\theta_{\mathfrak{O}}^{q-1}}$ is $R/(\sigma_1)$, $M_{\theta_{\mathfrak{O}}^{i}\otimes\theta_{\mathfrak{q}}^{1}} \cong R/(\sigma_2^2 - 4\sigma_1^p)$ for $0\leq i <q-1$ and $\theta_{\mathfrak{O}}^{q-1}\otimes\theta_{\mathfrak{q}}^{1} \cong R/(\Delta)$. These are all of the modules that correspond to $1$-dimensional representations.

All $2$-dimensional representations fall into one of the cases of Theorem \ref{mod odd}, although we must be careful, since some restricted representations are the same.

The representations Res$_H^{G}\rep{i}{jq-1}$, for $0\leq i <q-1$ case never fall into another case, since $i \neq q-1$.  Then for each $j$ there are $q-1$ representations of this form, giving rise to $q-1$ copies of $X_{p-j}$ in the decomposition of $S/(z)$.

Each representation in the $\mathrm{Res}^G_H\rep{q-1}{(j+1)q-1}$, for $ 1\leq j \leq p-1$, case is isomorphic to another one, for example: $\rep{q-1}{(j+1)q-1} \cong \rep{q-1}{m-(j-1)q-1}$. There are $p-1$ $m$-tuple of Young diagrams of the form $\rep{q-1}{(j+1)q-1}$ and each one is equivalent, under shifting, to exactly one distinct one of the same form. One can see that all representations  $\mathrm{Res}^G_H\rep{q-1}{(j+1)q-1}$ for $ 1\leq j \leq \frac{p-1}{2}$ are distinct and are all of the representations appearing from $m$-tuples of Young diagrams of the form  $\rep{q-1}{(j+1)q-1}$. Thus one copy of $N_j^2$ for $1\leq j\leq \frac{p-1}{2}$ appears in the decomposition.

For the case of the representations Res$^G_H\rep{i}{jq+r}$ for $0 \leq j \leq \frac{p-1}{2}$, $0 \leq i < q-1$ and $0 \leq r < q-1$ are all distinct and for $\frac{p-1}{2}< k<p-1$, $\rep{i}{kq+r}\cong \rep{r}{(p-k)q+i}$ where $0\leq p-k\leq \frac{p-1}{2}$. Each Res$^G_H\rep{i}{jq+r}$ for $0 \leq j \leq \frac{p-1}{2}$, $0 \leq i < q-1$ and $0 \leq r < q-1$, gives a copy of $K_j^2$ in the decomposition, and so by counting we get, for $j=0$ there are $q-1\choose2$ distinct representations - all giving $K_0$. For all $1 \leq j \leq \frac{p-1}{2}$ we have $(q-1)^2$ distinct representations coming from the choice of $i$ and $r$. These then each give a copy of $K_j^2$. 

This then completes the irreducible $2$-dimensional representation case, since all irreducible $2$-dimensional representations fall into one of the above cases, after shifting by $q$ and therefore are isomorphic.
\end{proof}
\subsection{\(m\neq p\) and \(p\) even.} 
For the case when $p$ is even, $\spec(R/(\Delta)$) is also a $D_{p+2}$ singularity. The modules above are still indecomposable $CM$ modules, although, since $p+2$ is even, there are a few more indecomposable $CM$ modules. Additionally we have

\begin{itemize}
    \item $C_+ = \cok(\sigma_1(\sigma_2+2\sigma_1^{\frac{p}{2}}),\sigma_2-2\sigma_1^{\frac{p}{2}}),$
    \item $D_+ = \cok(\sigma_2-2\sigma_1^{\frac{p}{2}},\sigma_1(\sigma_2+2\sigma_1^{\frac{p}{2}})),$
    \item $C_- = \cok(\sigma_1(\sigma_2-2\sigma_1^{\frac{p}{2}}),\sigma_2+2\sigma_1^{\frac{p}{2}}),$
    \item $D_- = \cok(\sigma_2+2\sigma_1^{\frac{p}{2}},\sigma_1(\sigma_2-2\sigma_1^{\frac{p}{2}})).$
\end{itemize}

We also have the isomorphisms $K_{\frac{p}{2}}\cong C_+ \oplus C_-$ and $N_{\frac{p}{2}}\cong D_+ \oplus D_-$
\begin{theorem}\label{mod even}
Let $H=G(m,p,2)$, with $p$ even. The discriminant defines a $D_{p+2}$ singularity and the 2-dimensional representations correspond to the following Cohen-Macaulay modules in the decomposition of $S/(z)$:

\begin{itemize}
    \item $M_{\mathrm{Res}^G_H\rep{i}{jq-1}}\cong X_{p+1-j}^2$ for $1 \leq j \leq p-1$ and $0\leq i <q-1$.
    \item $M_{\mathrm{Res}^G_H\rep{0}{jq}}\cong K_j^2$ for $1 \leq j \leq p-1$, $j\neq\frac{p}{2} $.
    \item $M_{\mathrm{Res}^G_H\rep{q-1}{(j+1)q-1}} \cong N_j^2$ for $1 \leq j \leq p-1$, $j\neq\frac{p}{2} $.
\end{itemize}
\end{theorem}
\begin{proof} The calculations are the same as in the odd case.
\end{proof} 
\begin{remark} From Section \ref{lin}, the linear characters give the modules corresponding to the components of the discriminant. \end{remark}
\begin{theorem}
Let $H=G(m,p,2)$, with $p$ even, $m\neq p$ and $q=\frac{m}{p} $ Then:
\begin{equation*}
\begin{split}
S/(z) &\cong \bigoplus_{j=1}^{p}X^{2(q-1)}_j \oplus \bigoplus_{j=1}^{\frac{p-2}{2}} N_j^2 \oplus \bigoplus_{j=1}^{\frac{p-2}{2}}K_j^{2(q-1)^{2}} \oplus K_{\frac{p}{2}}^{2 {q-1 \choose 2}}\oplus K_0^{2{q-1 \choose 2}} \oplus R/\left(\sigma_1\right)\oplus (R/\left(\sigma_2^2 - 4\sigma_1^p\right))^{q-1}   \\ & \hspace{0.5cm}\oplus C_+ \oplus D_+ \oplus C_- \oplus D_-\oplus R/(\Delta)
 \\ &\cong \bigoplus_{j=1}^{p}X^{2(q-1)}_j \oplus \bigoplus_{j=1}^{\frac{p-2}{2}} N_j^2 \oplus \bigoplus_{j=1}^{\frac{p-2}{2}}K_j^{2(q-1)^{2}}  \oplus R/\left(\sigma_1\right)\oplus (R/\left(\sigma_2^2 - 4\sigma_1^p\right))^{q-1 + 2 {q-1 \choose 2 }} 
 \\&\hspace{0.5cm} \oplus C_+^{1+ 2 {q-1 \choose 2}} \oplus D_+ \oplus C_-^{1+2 {q-1 \choose 2}} \oplus D_-\oplus R/(\Delta)
\end{split}
\end{equation*}
\end{theorem}

\begin{proof}
Let $\mathfrak{O}$ be the orbit of hyperplanes such that $j_\mathfrak{O}:=\prod_{H\in \mathfrak{O}}\alpha_H= \sigma_1$ and let $\mathfrak{q}_+, \mathfrak{q}_-$ be the orbit such that $j_{\mathfrak{q_+}}= \sigma_2+2\sigma_1^{\frac{p}{2}}, j_{\mathfrak{q_-}}= \sigma_2-2\sigma_1^{\frac{p}{2}}$. The linear characters are given by $\theta_{\mathfrak{O}}^{i}\otimes\theta_{\mathfrak{q}_+}^{j} \otimes \theta_{\mathfrak{q}_-}^{k}$ for $0\leq i \leq q-1$, $0\leq j \leq 1$ and $0\leq k \leq 1$. The modules $M_{\theta_{\mathfrak{O}}^{i}}$ for $0\leq i <q-1$ give the trivial module. The module $M_{\theta_{\mathfrak{O}}^{q-1}}$ is $R/(\sigma_1)$, $M_{\theta_{\mathfrak{O}}^{i}\otimes\theta_{\mathfrak{q_+}}^{1}\otimes \theta_{\mathfrak{q_-}}^1} \cong R/\sigma_2^2 - 4\sigma_1^p$ for $0\leq i <q-1$ and \mbox{$M_{\theta_{\mathfrak{O}}^{q-1}\otimes\theta_{\mathfrak{q_+}}^{1}\otimes \theta_{\mathfrak{q_-}}^1} \cong R/(\Delta)$}. Now we can find the isomorphisms $M_{\theta_{\mathfrak{O}}^{q-1}\otimes\theta_{\mathfrak{q_+}}^{1}} \cong C_+$ and $M_{\theta_{\mathfrak{O}}^{i}\otimes\theta_{\mathfrak{q_-}}^{1}} \cong D_+$ for $0 \leq i <q-1$. Also $M_{\theta_{\mathfrak{O}}^{q-1}\otimes\theta_{\mathfrak{q_-}}^{1}} \cong C_-$ and  $M_{\theta_{\mathfrak{O}}^{i}\otimes\theta_{\mathfrak{q_+}}^{1}} \cong D_-$ for $0 \leq i <q-1$. 

All $2$-dimensional representations fall into one of the cases of Theorem \ref{mod even}, although we must be careful, since some of the representations considered in the cases above are isomorphic to each other.

The representations  $\mathrm{Res}^G_H\rep{i}{jq-1}$, for $0\leq i <q-1$ case are never fall into another case, since $i \neq q-1$.  Then for each $j$ there are $q-1$ representations of this form, giving rise to $q-1$ copies of $X_j$ in the decomposition of $S/(z)$.

Each representation in the $\mathrm{Res}^G_H\rep{q-1}{(j+1)q-1}$, for $ 1\leq j \leq p-1$ and $j \neq \frac{p}{2}$, case is isomorphic to another one, for example: $\rep{q-1}{(j+1)q-1} \cong \rep{q-1}{m-(j-1)q-1}$. When $j = \frac{p}{2}$, $\mathrm{Res}^G_H\rep{q-1}{(j+1)q-1}$ is not an irreducible $2$-dimensional representation, since it is the direct sum of two $1$-dimensional ones. There are $(p-1)$ $m$-tuples of Young diagrams of the form $\rep{q-1}{(j+1)q-1}$ and each one is equivalent, under shifting, to exactly one distinct one of the same form. One can see that all representations  $\mathrm{Res}^G_H\rep{q-1}{(j+1)q-1}$ for $ 1\leq j \leq \frac{p-2}{2}$ are distinct and are all of the representations appearing from $m$-tuples of Young diagrams of the form  Res$^G_H\rep{q-1}{(j+1)q-1}$. Thus one copy of $N_j^2$ for $1\leq j\leq \frac{p-1}{2}$ appears in the decomposition.

For the case of the representation Res$^G_H\rep{i}{jq+r}$ for $0 \leq j \leq \frac{p-2}{2}$, $0 \leq i < q-1$ and $0 \leq r < q-1$ are all distinct and for $\frac{p+2}{2}\leq  k<p-1$, $\rep{i}{kq+r}\cong \rep{r}{(p-k)q+i}$ where $0\leq p-k\leq \frac{p-2}{2}$. Each Res$^G_H\rep{i}{jq+r}$ for $0 \leq j \leq \frac{p-2}{2}$, $0 \leq i < q-1$ and $0 \leq r < q-1$ gives a copy of $K_j^2$ in the decomposition, and so by counting we get, for $j=0$ there are $q-1\choose2$ distinct representations - all giving $K_0$. For all $1 \leq j \leq \frac{p-2}{2}$ we have $(q-1)^2$ distinct representations coming from the choice of $i$ and $r$ - each giving a copy of $K_j^2$. Since we are in the even case if $i=r$ and $j=\frac{p}{2}$ then the representation Res$^G_H\rep{i}{jq+r}$ splits into two irreducible $1$-dimensional representations, but if $i \neq r$ and $j=\frac{p}{2}$ then Res$^G_H\rep{i}{jq+r}$ is an irreducible $2$-dimensional representation and gives a copy of $K_{\frac{p}{2}}^2$. Note that for $i \neq r$ and $j=\frac{p}{2}$, $\rep{i}{jq+r}\cong\rep{r}{jq+i} $ and so the number of distinct representations of this form are counted by ${q-1 \choose 2}$.

This then completes the irreducible $2$-dimensional representation case, since all irreducible $2$-dimensional representations fall into one of the above cases, after shifting.
\end{proof}
\subsection{\( m=p \)}
The case of $G(m,m,2)$ with $m>2$ is similar to that of $G(2p,p,2)$; $G(m,m,2)$ is also true reflection group and thus we already know \cite{BFI} that $S/(z)$ is an $NCR$ for $R/(\Delta)$. In loc. cit. it was shown that for a true reflection group in dimension 2 there is a $1-1$ correspondence between the isomorphism classes of indecomposible Cohen-Macaulay modules over the discriminant and the isotypical components. We calculate the correspondence using the \cite{YujiCM} for the matrix factorization for the discriminant, which in this case is an $A_{p-1}$ singularity. Let $G=G(m,1,2)$ with $m>2$ and $H=G(m,m,2)$, the representations Res$^G_H\rep{i}{j}\cong$Res$^G_H\rep{i+1}{j+1}$ where $0\leq i < j < m$ and $i+1$ and $j+1$ are taken mod $m$. From this we see that the action of $-\otimes \det$ is the identity on all irreducible two-dimension representations of $H$.

\begin{remark}
We avoid the case $m=2$ since $G(2,2,2)$ is a reducible reflection group and the case $m=1$ is $S_2$.
\end{remark}

The 2-dimensional irreducible representations are then distinguished by the distance between the two boxes in the tuple. In particular the representations Res$^G_H\rep{0}{i}$ for $1\leq i <\frac{m}{2}$ are all distinct and any other irreducible two-dimensional representation is isomorphic to one of these, this can be seen by shifting the $m$-tuple of diagrams. Note that if $m$ is even the representation Res$^G_H \rep{0}{\frac{m}{2}}$ is not irreducible. 

Recalling that the discriminant is $\Delta=(x^m-y^m)^2=\sigma_2^2-4\sigma_1^m$, then Spec$(R/(\Delta))$ is an $A_{m-1}$ singularity. From \cite{YujiCM} the $CM$-modules come from the matrix factorizations $(\phi_j,\psi_j)$ where:

\[\phi_j= \begin{bmatrix}\sigma_2 & 2\sigma_1^{m-j} \\ 2\sigma_1^j & \sigma_2 \end{bmatrix},\psi_j= \begin{bmatrix}\sigma_2 & -2\sigma_1^{m-j} \\ -2\sigma_1^j & \sigma_2 \end{bmatrix} \text{ For } 0\leq j\leq p\]

Let $X_j=\textrm{Coker}(\phi_j,\psi_j)$, then $X_0\cong R$, $X_j\cong X_{p-j}$. When $m$ is odd, these are all the irreducible Cohen-Macaulay modules over $R/(\Delta)$.  
When $m$ is even, then $(\sigma_2+2\sigma_1^{\frac{m}{2}},\sigma_2-2\sigma_1^{\frac{m}{2}})$ and  $(\sigma_2-2\sigma_1^{\frac{m}{2}},\sigma_2+2\sigma_1^{\frac{m}{2}})$ are matrix factorizations. Let $N_+$ and $N_-$ be the $CM$ modules given by $(\sigma_2+2\sigma_1^{\frac{m}{2}},\sigma_2-2\sigma_1^{\frac{m}{2}})$ and $(\sigma_2-2\sigma_1^{\frac{m}{2}},\sigma_2+2\sigma_1^{\frac{m}{2}})$ respectively, Then $X_{\frac{m}{2}}\cong N_+ \oplus N_-$.
\begin{lemma}
Let $G=G(m,1,2)$ and $H=G(m,m,2)$ then the modules $M_{\textrm{Res}^G_H \rep{0}{i}}\cong X_i^2$ for $0\leq i < \frac{m}{2}$
\end{lemma}\begin{proof}

A basis for the isotypical components of the coinvariant algebra $S_H$ of $H$, of type Res$^G_H\rep{0}{i}$ for $1\leq i < \frac{m}{2}$ is $\{x^i,y^i,x^{m-i},y^{m-i}\}$. Recalling that $z=x^m-y^m$, calculating on this basis, multiplication by $z$ can be expressed as the matrix:

\[\begin{pmatrix}
\sigma_2 & 0 & 0 & 2\sigma_1^{m-i} \\ 0&-\sigma_2 &-2\sigma_1^{m-i}&0 \\ 0&2\sigma_1^i&\sigma_2&0 \\ -2\sigma_1^i &0&0&-\sigma_2y^{m-i}
\end{pmatrix}\]

This is then equivalent as matrix factorizations to:

\[\begin{tikzcd}[ampersand replacement=\&, column sep =2.5cm]
S^G_{\mathrm{res}^G_H \rep{0}{i}}\arrow{r}{j} \& S^G_{\mathrm{res}^G_H \rep{0}{i}} \arrow{r}{ \begin{pmatrix}\sigma_2 & 2\sigma_1^{m-i} \\ 2\sigma_1^{i} & \sigma_2\end{pmatrix} \otimes I_2} \& S^G_{\mathrm{Res}^G_H \rep{0}{i}} 
\end{tikzcd}\]
\end{proof}
So in the odd case we have found all the indecomposable $CM$ modules over $R/(\Delta)$. In the even case since Res$^G_H\rep{0}{\frac{m}{2}}$ splits into 2 one-dimensional representations, from Section \ref{lin} we obtain the modules corresponding to the components of the discriminant.

\begin{theorem}\label{Theorem:Full}
Let $H=G(m,m,2)$, with $m$ even, Then:
$$S/(z) \cong \bigoplus_{j=1}^{\frac{m-2}{2}} X_j^2 \oplus N_+ \oplus N_-\oplus R/(\Delta). $$

Let $H=G(m,m,2)$, with $m$ odd, Then:
$$S/(z) \cong \bigoplus_{j=1}^{\frac{m-1}{2}} X_j^2 \oplus R/(\Delta). $$
\end{theorem}
\begin{proof} The representations $\rep{0}{j}$ are distinct for $1\leq j \leq \frac{m-1}{2}$ and $\rep{0}{\frac{m-1}{2}+i}\cong \rep{0}{\frac{m-1}{2}-i}$ for $1 \leq i \leq \frac{m-1}{2}$. One copy of $N_+$, $N_-$ and $R/(\Delta)$ come from the 1-dimensional representations.  \end{proof}
\subsection{Crossover with true reflection groups.}

 The groups $G(2p,p,2)$ are true reflection groups and so by \cite{BFI} there is a $1-1$ correspondence between the irreducible representations and isomorphism classes of indecomposible $CM$-modules over the discriminant. This section details the 1-1 correspondence using the notation from before. We only present the $p$ odd case, since the $p$ even case is similar.

\begin{lemma}
The irreducible two-dimensional representations of $H=G(2p,p,2)$ where $p$ is odd are given by:
\begin{itemize}
\item[1)] $\mathrm{Res}^G_H\rep{0}{2i-1}$ for $1\leq i \leq p$.
    \item[2)] $\mathrm{Res}^G_H\rep{0}{2i}$ for $1\leq i \leq \frac{p-1}{2}$.

    \item[3)] $\mathrm{Res}^G_H\rep{1}{2i+1}$ for $1 \leq i \leq \frac{p-1}{2}$.
\end{itemize}
\end{lemma}
\begin{proof}
Since for a given $2p$-tuple of young diagrams we can always shifted until there is a box in the first $2$ positions, it is enough to consider the $2p$-tuples such that there is a box in position $0$ or position $1$. This gives us two cases:

\begin{itemize}
   \item[1)] $\mathrm{Res}^G_H\rep{0}{i}$ for $1\leq i \leq 2p$.
\item[2)] $\mathrm{Res}^G_H\rep{1}{i}$ for $1\leq i \leq 2p$.

\end{itemize}
\noindent Then noting that $\rep{0}{2i-1}\cong\rep{1}{2p-2(i-1)}$ gives us the list above.
\end{proof}

Recall that for $G(2p,p,2)$, the discriminant Spec($R/(\Delta)$) is a $D_{p+2}$ singularity. We use the same notation for the modules as in Theorem \ref{mod odd}

Let $p$ be odd, for the subgroup $H=G(2p,p,2)$ of $G=(2p,1,2)$ the following holds:
\begin{itemize}
    \item[1)] $M_{\mathrm{Res}^G_H\rep{0}{2i-1}} \cong X_{i}^2$for $1\leq i \leq p$.
    \item[2)]$M_{\mathrm{Res}^G_H\rep{0}{2i}}\cong K_i^2$ for $1\leq i \leq \frac{p-1}{2}$.
    \item[3)] $M_{\mathrm{Res}^G_H\rep{1}{2i+1}}\cong N_i^2 $ for $1 \leq i \leq \frac{p-1}{2}$.
\end{itemize}
From this we can see the $1-1$ correspondence on the $2$-dimensional representations, since for the cases above there is only one representation of each form.

\begin{example}
Consider $H=G(6,3,2)$ and $G=G(6,1,2)$ then there are the following equivalence classes of $6$-tuples of young diagrams:
\begin{equation*}
    \begin{split}
        \left\{\,\rep{0}{1},\rep{2}{3}, \rep{4}{5}\right\} \hspace{1cm} & \hspace{1cm} \left\{\,\rep{0}{2},\rep{2}{4},\rep{0}{4}\right\} \\
        \left\{\,\rep{0}{3},\rep{2}{5}, \rep{1}{4}\right\} \hspace{1cm} & \hspace{1cm} \left\{\,\rep{1}{2},\rep{3}{4},\rep{0}{5}\right\}  \\
        \left\{\,\rep{1}{3},\rep{3}{5},\rep{1}{5}\right\} \hspace{1cm} &
    \end{split}
\end{equation*}
Then $M_{\mathrm{Res}^G_H\rep{0}{1}}=M_{\mathrm{Res}^G_H\rep{2}{3}}\cong X_1^2$, $M_{\mathrm{Res}^G_H\rep{0}{2}}\cong K_1^2$, $M_{\mathrm{Res}^G_H\rep{0}{3}}\cong X_2^2$,  $M_{\mathrm{Res}^G_H\rep{1}{3}}\cong N_1^2$, and  $M_{\mathrm{Res}^G_H\rep{1}{2}}\cong Y_1^2(\cong M_{\mathrm{Res}^G_H\rep{0}{5}}\cong X_3^2$). The linear characters are of the form ${\mathrm{Res}^G_H\Yvcentermath1 \yng(1,1)}_0$,${\mathrm{Res}^G_H\Yvcentermath1 \yng(1,1)}_1 $, ${\mathrm{Res}^G_H\Yvcentermath1 \yng(2)}_0$ and ${\mathrm{Res}^G_H\Yvcentermath1 \yng(2)}_1$. The modules corresponding to these representations; $M_{\mathrm{Res}^G_HM_{{\Yvcentermath1 \yng(2)}_0}}$ is trivial, $M_{\mathrm{Res}^G_HM_{{\Yvcentermath1 \yng(1,1)}_1}}\cong R/(\Delta)$,  $M_{\mathrm{Res}^G_HM_{{\Yvcentermath1 \yng(1,1)}_0}}$ and $M_{\mathrm{Res}^G_HM_{{\Yvcentermath1 \yng(2)}_1}}$ are the two irreducible components of $\Delta$. This shows the 1-1 correspondence between irreducible representations of $G$ and irreducible $CM$-modules over the discriminant.
Non-Commutative Resolutions for the Discriminant of the Complex Reflection group $G(m,p,2)$

\end{example}
\bibliographystyle{alpha}
\bibliography{mybib}{}
\end{document}